\documentclass[11pt]{amsart}
\usepackage[foot]{amsaddr}
\usepackage[english]{babel}

\allowdisplaybreaks

\usepackage[
pdftex,hyperfootnotes]{hyperref}
\hypersetup{
	colorlinks=true,
	linkcolor=NavyBlue, 
	urlcolor=RoyalPurple,
	citecolor=OliveGreen,
	pdftitle={Oscillory banded matrices and mixed multiple orthogonality},
	bookmarks=true,
	pdfpagemode=FullScreen,
}
\usepackage[usenames,dvipsnames,svgnames,table,x11names]{xcolor}
\usepackage{pgfplots}
\usepackage{tikz}
\usepackage{tikz-3dplot}
\usetikzlibrary{automata,quotes, chains,matrix,calc,shadows,shapes.callouts,shapes.geometric,shapes.misc,positioning,patterns,decorations.shapes,
	decorations.pathmorphing,decorations.markings,decorations.fractals,decorations.pathreplacing,shadings,fadings,arrows.meta,bending}

\usepackage{nicematrix}
\usepackage[utf8]{inputenc}

\usepackage{comment}
\usepackage[textwidth=17.5cm,textheight=22.275cm,
height=
24.275cm,
width=18cm]{geometry}

\usepackage{amssymb,latexsym,amsmath,amsthm,bm}
\usepackage{mathrsfs}
\usepackage{mathtools,arydshln,mathdots}

 \usepackage{tcolorbox}

\usepackage[table]{xcolor}

\usepackage{Baskervaldx}
\usepackage[]{newtxmath}

\usepackage{bigints}

\theoremstyle{plain}

\newtheorem{teo}{Theorem}[section]
\newtheorem{coro}[teo]{Corollary}
\newtheorem{lemma}[teo]{Lemma}
\newtheorem{pro}[teo]{Proposition}

\theoremstyle{definition}
\newtheorem{defi}[teo]{Definition}

\theoremstyle{remark}
\newtheorem{rem}[teo]{Remark}

\renewcommand{\d}{\operatorname{d}}

\newcommand{\sgn}{\operatorname{sgn}}

\newcommand{\diag}{\operatorname{diag}}

\newcommand{\C}{\mathbb{C}}

\newcommand{\A}{\mathbb{A}}
\newcommand{\N}{\mathbb{N}}
\newcommand{\R}{\mathbb{R}}

\allowdisplaybreaks[1]

\makeatletter
\DeclareRobustCommand{\gaussk}{\DOTSB\gaussk@\slimits@}
\newcommand{\gaussk@}{\mathop{\vphantom{\sum}\mathpalette\bigcal@{K}}}

\newcommand{\bigcal@}[2]{%
	\vcenter{\m@th
		\sbox\z@{$#1\sum$}%
		\dimen@=\dimexpr\ht\z@+\dp\z@
		\hbox{\resizebox{!}{0.8\dimen@}{$\mathcal{K}$}}%
	}%
}
\newcommand{\cfracplus}{\mathbin{\cfracplus@}}
\newcommand{\cfracplus@}{%
	\sbox\z@{$\dfrac{1}{1}$}%
	\sbox\tw@{$+$}%
	\raisebox{\dimexpr\dp\tw@-\dp\z@\relax}{$+$}%
}
\newcommand{\cfracdots}{\mathord{\cfracdots@}}
\newcommand{\cfracdots@}{%
	\sbox\z@{$\dfrac{1}{1}$}%
	\sbox\tw@{$+$}%
	\raisebox{\dimexpr\dp\tw@-\dp\z@\relax}{$\cdots$}%
}
\makeatother

\makeatletter
\newcommand*{\relrelbarsep}{.386ex}
\newcommand*{\relrelbar}{%
	\mathrel{%
		\mathpalette\@relrelbar\relrelbarsep
	}%
}
\newcommand*{\@relrelbar}[2]{%
	\raise#2\hbox to 0pt{$\m@th#1\relbar$\hss}%
	\lower#2\hbox{$\m@th#1\relbar$}%
}
\providecommand*{\rightrightarrowsfill@}{%
	\arrowfill@\relrelbar\relrelbar\rightrightarrows
}
\providecommand*{\leftleftarrowsfill@}{%
	\arrowfill@\leftleftarrows\relrelbar\relrelbar
}
\providecommand*{\xrightrightarrows}[2][]{%
	\ext@arrow 0359\rightrightarrowsfill@{#1}{#2}%
}
\providecommand*{\xleftleftarrows}[2][]{%
	\ext@arrow 3095\leftleftarrowsfill@{#1}{#2}%
}
\makeatother
\begin{document}
	
	\title[Oscillation and mixed multiple orthogonality]{Spectral theory for bounded banded matrices\\ with positive bidiagonal factorization and \\mixed multiple orthogonal polynomials}

	\author[A Branquinho]{Amílcar Branquinho$^{1}$}
	\address{$^1$CMUC, Departamento de Matemática,
		Universidade de Coimbra, 3001-454 Coimbra, Portugal}
	\email{$^1$ajplb@mat.uc.pt}
	
	\author[A Foulquié]{Ana Foulquié-Moreno$^{2}$}
	\address{$^2$CIDMA, Departamento de Matemática, Universidade de Aveiro, 3810-193 Aveiro, Portugal}
	\email{$^2$foulquie@ua.pt}
	
	\author[M Mañas]{Manuel Mañas$^{3}$}
	\address{$^3$Departamento de Física Teórica, Universidad Complutense de Madrid, Plaza Ciencias 1, 28040-Madrid, Spain \&
		Instituto de Ciencias Matematicas (ICMAT), Campus de Cantoblanco UAM, 28049-Madrid, Spain}
	\email{$^3$manuel.manas@ucm.es}




	\begin{abstract}
		Spectral and factorization properties of oscillatory matrices leads to a spectral Favard theorem for bounded banded matrices, that admit a positive bidiagonal factorization, in terms of sequences of mixed multiple orthogonal polynomials with respect to a set positive Lebesgue--Stieltjes measures.
	A mixed multiple Gauss quadrature formula with  corresponding degrees of precision is given. 
			\end{abstract}
	
	\maketitle
	
	
	 \tableofcontents
	
	\section{Introduction}

This work is devoted to spectral theorems beyond self-adjoint or normal operators. We give conditions, the existence of a positive bidiagonal factorization, to be explained later, such we can state a spectral Favard theorem for bounded banded semi-infinite matrices. The study of symmetric tridiagonal operators acting in the Hilbert space $\ell^2$ can under an appropriate chosen basis be reduced to an infinite Jacobi matrix enabling a deeper understanding of this spectral theory.
Here is where the theory of general orthogonal polynomials come into place to derive the spectral and the resolvent set for selfadjoint operators (cf. \cite{nikishin_sorokin}).

 Multiple orthogonal polynomials are traditionally linked with the theory of Hermite--Pad\'e and its applications to the constructive function theory. For  good introductions to multiple orthogonal polynomials see the book by Nikishin and Sorokin \cite{nikishin_sorokin} and the chapter by Van Assche in \cite[Ch. 23]{Ismail} and for a  inspiring basic introduction see \cite{andrei_walter}.  Multiple orthogonal polynomials is a very active research area: for asymptotic of zeros see \cite{Aptekarev_Kaliaguine_Lopez}, for a Gauss--Borel perspective see \cite{afm}, for Christoffel perturbations see \cite{bfm}, for applications to random matrix theory see \cite{Bleher_Kuijlaars}. Mixed multiple orthogonal polynomials, and the corresponding Riemann--Hilbert problem, have found applications in Brownian bridges, or non-intersecting Brownian motions, that leave from $ p$ points and arrive to $q$ points \cite{Evi_Arno}, and in the discussion of multicomponent Toda, cf.  \cite{adler,afm}. Mixed multiple orthogonal polynomials  also appear in applications to number theory.  Apéry, cf. \cite{Apery}, proved in 1979 that $\zeta(3)$ is irrational. The proof is based  on an mixed Hermite--Padé approximation to three functions. Also mixed multiple orthogonal polynomials and corresponding mixed Hermite--Padé  approximations have been used  to show that infinitely many values of the $\zeta$ function at odd integers are irrational, cf. \cite{Ball_Rivoal}, and  that at least one of the numbers $ \zeta(5),\zeta(7),\zeta(9),\zeta(11)$ is irrational, cf.  \cite{Zudilin}.

Another field of application of multiple orthogonal polynomials is the spectral analysis of high order difference operators.
Spectral theorems hold beyond self-adjointness for normal operators (the operator commutes with its adjoint). In the case of banded Hessenberg operators, the self-adjointness or normality no more takes place.

A first attempt to tackle these problems has been made by Kalyagin in \cite{Kalyagin,Kaliaguine,KaliaguineII,Aptekarev_Kaliaguine_Saff}. There the author defines a class of operators related with the Hermite--Pad\'e approximants and connects their spectral analysis with the convergence problem for simultaneous Hermite--Pad\'e rational approximants
of a system of resolvent functions of the operator 
(that coincides with the notion of multiple orthogonality). See also \cite{Beckermann_Osipov}.
At the same time this connection leads to a new solution of the direct and inverse spectral problems for the operators based on the Jacobi--Perron algorithm and vector continued fractions (cf. \cite{Aptekarev_Kaliaguine_VanIseghem,Sorokin_Van_Iseghem_1,Sorokin_Van_Iseghem_2}).
This approach serves as a base of a new method of the investigation on nonlinear discrete dynamical systems. As an example, global solutions of a hierarchy of discrete KdV equations are obtained (cf. \cite{Aptekarev_Kaliaguine_VanIseghem,Sorokin_Van_Iseghem_3,dolores_ana_ab1,dolores_ana_ab2,dolores}).

Recently, in a series of works (cf. \cite{nuestro2,PBF_1,nuestro1}) we have analyzed the applications of type I and II multiple orthogonal polynomials to certain Markov chains also called non simple random walks (i.e., beyond birth and death). At the end we got a spectral Favard theorem with an application to Markov chains described in terms of a bounded banded ($p+2$)-diagonal (with one superdiagonal and $p$ subdiagonals) oscillatory Hessenberg operators that admit positive bidiagonal factorizations.

\enlargethispage{1cm}
The main result we achieved in \cite{PBF_1} is that bounded banded Hessenberg matrices that admit a positive bidiagonal factorization have a set of positive Stieltjes--Lebesgue measures, and can be spectrally described by multiple orthogonal polynomials. This extends to the non-normal scenario the spectral Favard theorem for Jacobi matrices (cf. \cite{Ismail}).
An important feature of the method applied in \cite{PBF_1} to describe the spectrality of a banded Hessenberg operator, is the multiple Gauss quadrature formula that we get with the exact degrees of precision.
	
	In this paper we consider a bounded banded operator $T$ whose semi-infinite matrix is a banded matrix with $q$ superdiagonals and $p$ subdiagonals and such that the leading principal
	submatrices are given by
	\begin{align}\label{eq:monic_Hessenberg}
		T^{[N]}&=
		\begin{bNiceMatrix}[columns-width = 1.5cm]
			T_{0,0} &\Cdots &&T_{0,q}& 0 & \Cdots &&0\\
			\Vdots& &&&\Ddots &\Ddots& &\Vdots\\
			T_{p,0}&&&&&&&\\
			0&\Ddots&&&&&&0\\
			\Vdots&\Ddots&&&&&&T_{N-q,N}\\
			&&&&&&&\Vdots\\
			&&&&&&&\\[6pt]
			0&\Cdots&&&0&T_{N,N-p}&\Cdots&T_{N,N}
		\end{bNiceMatrix}
	\end{align}
	where  it assumed that the entries in the extreme diagonals do not cancel
\begin{align}\label{eq:not_zero_extreme_diagonals}
		T_{n+p,n}&\neq 0,&T_{n,n+q}&\neq 0, & n&\in\N_0.
\end{align} 
In what follows we will show that when this  matrix, after a shift, admits a positive bidiagonal factorization and, consequently, is oscillatory, we can find a spectral Favard theorem. Now, we have mixed multiple orthogonal polynomials with respect to a matrix of positive Lebesgue--Stieltjes measures. As an application we derive the corresponding Gauss quadrature formula for this matrix of measures and determine their degrees of precision.

While the extension given in this paper  of the spectral Favard theorem of \cite{PBF_1} for banded Hessenberg matrices to arbitrary banded matrices  is natural, there were several key issues to resolve before achieving this large extension. The first one was to understand the role of the characteristic polynomial that is no longer an orthogonal polynomial, the second was to find the relationship of the determinants of the two families of mixed multiple orthogonal polynomials to the characteristic polynomial and its zeros, and finally the extension of the positive bidiagonal factorization  to this general situation and the application of the oscillation properties of eigenvectors.
	
	Within this introduction we discuss, in the first place, some preliminary staff on totally nonnegative  matrices, stating (without proof) the results needed later on. Then, we show how the well known bounded tridiagonal Jacobi matrix for which we have the spectral Favard theorem happens to be oscillatory after an adequate shift and have a positive bidiagonal factorization. Finally, we use the Gauss--Borel factorization of a moment matrix to construct mixed multiple orthogonality on the step-line.

	\subsection{Totally nonnegative and oscillatory matrices}
Totally nonnegative (TN) matrices are those with all their minors nonnegative, cf.  \cite{Fallat-Johnson,Gantmacher-Krein}, and the set of nonsingular TN matrices is denoted by InTN. Oscillatory matrices, cf.  \cite{Gantmacher-Krein}, are totally nonnegative, irreducible \cite{Horn-Johnson} and nonsingular. Notice that the set of oscillatory matrices is denoted by IITN (irreducible invertible totally nonnegative) in \cite{Fallat-Johnson}. An oscillatory matrix~$T$ is equivalently defined as a totally nonnegative matrix $A$ such that for some $n$ we have that $A^n$ is totally positive (all minors are positive). From Cauchy--Binet Theorem one can deduce the invariance of these sets of matrices under the usual matrix product. Thus, following \cite[Theorem 1.1.2]{Fallat-Johnson} the product of matrices in InTN is again InTN (a similar statement hold for TN or oscillatory matrices).

We have the important result:
\begin{teo}[Gantmacher--Krein Criterion] \cite[Chapter 2, Theorem 10]{Gantmacher-Krein}. \label{teo:Gantmacher--Krein Criterion}
	A~totally non negative matrix  is oscillatory if and only if it is nonsingular and the elements at the first subdiagonal and first superdiagonal are positive.
\end{teo}

Regarding tridiagonal matrices we have the following classical result:
\begin{teo}\label{teo:Jacobi oscillatory}
	\cite[Chapter XIII,\S 9]{Gantmacher} and \cite[Chapter 2,Theorem 11]{Gantmacher-Krein}.
	A tridiagonal matrix is oscillatory if and only if, 
	\begin{enumerate}
		\item The matrix entries of the first subdiagonal and first superdiagonal are positive.
		\item All leading principal minors are positive.
	\end{enumerate}
\end{teo}

Gauss--Borel factorizations are intimately related with these concepts:

\begin{teo}\cite[Theorem 2.4.1]{Fallat-Johnson}\label{teo:LU}
$T\in\operatorname{InTN}$  if and only if it admits a Gauss--Borel factorization $T=L^{-1}U^{-1}$ with $L,U\in\operatorname{InTN}$, lower and upper triangular matrices, respectively.
\end{teo}

The following spectral theorems are extracted from \cite{Fallat-Johnson}, see also \cite{Gantmacher-Krein}.

\begin{teo}[Eigenvalue] \cite[Theorem 5.2.1]{Fallat-Johnson}\label{teo:eigennvalues_TN} 
	Given an oscillatory matrix $T\in\R^{N\times N}$~ the eigenvalues of $T$ are $N$ distinct positive numbers. 
\end{teo}

\begin{teo}[Interlacing of eigenvalues]\cite[Theorem 5.5.2]{Fallat-Johnson}\label{teo:eigennvalues_TN_2} 
Given an oscillatory matrix $T\in\R^{N\times N}$~ the eigenvalues of $T$ strictly interlace the eigenvalues of the two principal submatrices of order $(N-1)$, $T(1)$ or $T(N)$, obtained from $T$ by deleting the first row and column or the last row and column. 
\end{teo}

We  need to introduce the following notation.
We define the total sign variation of a totally nonzero vector (no entry of the vector $u$ is zero) as $v(u)=\text{cardinal}\{i\in\{1,\dots,n-1\} : u_iu_{i+1}<0\}$.
For a general vector $u\in\R^n$ we define $v_m(u)$ ($v_M(u)$) as the minimum (maximum) value $v(y)$ among all totally nonzero vectors $y$ that coincide with $u$ in its nonzero entries. For $v_m(u)=v_M(u)$ we write $v(u)\coloneq v_m(u)=v_M(u)$.

\begin{teo}[Eigenvectors] 
	\label{teo:eigenvectorII}
	Let $T\in\R^{N\times N}$ be an oscillatory matrix, and $u^{(k)}$ ($w^{(k)}$) the right (left) eigenvector corresponding to $ \lambda_k$, the $k$-th largest eigenvalue of $A$. Then
	\begin{enumerate}
		\item \cite[Theorem 5.3.3]{Fallat-Johnson} We have $v_m(u^{(k)})=v_M(u^{(k)})= v(u^{(k)})=k-1$ ($v_m(w^{(k)})=v_M(w^{(k)})= v(w^{(k)})=k-1$). Moreover, the first and last entry of $u^{(k)}$ ($w^{(k)})$ are nonzero, and $u^{(1)}$ and $u^{(N)}$ ($w^{(1)}$ and $w^{(N)}$) are totally nonzero; the other vectors may have a zero entry.
		\item From Perron--Frobenius theorem we know that $u^{(1)}$ ($w^{(1)}$) can be chosen to be entrywise positive, and that the other eigenvectors $u^{(k)}$ ($w^{(k)}$), $k=2,\dots,n$ have at least one entry sign change. In fact, $u^{(N)}$ ($w^{(N)}$) strictly alternates the sign of its entries.
		
	\end{enumerate}
\end{teo}

\subsection{Jacobi matrices}
Let us consider the tridiagonal semi-infinite real matrix
\begin{align}\label{eq:Jacobi_matrix_def}
	J&\coloneq \left[\begin{NiceMatrix}[columns-width = auto]
		\mathscr m_0& 1 &0&\Cdots&\\
		\mathscr l_1 &\mathscr m_1& 1& \Ddots&\\
		0&\mathscr l_2&\mathscr m_2&1&\\
		\Vdots& \Ddots& \Ddots & \Ddots&\Ddots\\
		&&&&\\
	\end{NiceMatrix}\right],
\end{align}
and assume that $\mathscr l_k>0$, $k\in \N$. This matrix is symmetrizable, as the positive diagonal matrix 
$	H=\diag(H_0,H_1,\dots)$, $H_0=1$, $H_n\coloneq \mathscr l_1\cdots \mathscr l_{n}$,
is such that $H^{-\frac{1}{2}}JH^{\frac{1}{2}}$ is symmetric.

If the matrix $J$ is bounded, all the possible eigenvalues of the submatrices $J^{[N]}$ belong to the disk $D(0,\|J\|)$.
As all the eigenvalues are real, let us consider those that are negative, and let $b$ be the supreme of the absolute values of all negative eigenvalues. Notice that $b \leqslant \|J\|$. 
\begin{teo}
	For $s\geqslant b$, the matrix $J_s=J+sI$ is oscillatory and admits a positive bidiagonal factorization in the form
	\begin{align*}
		J&=	\left[\begin{NiceMatrix}[columns-width=auto]
			1 &0&\Cdots&\\
			\alpha_2& 1 &\Ddots&\\
			0& \alpha_4& 1& \\
			\Vdots&\Ddots& \Ddots& \Ddots\\&&&
		\end{NiceMatrix}\right]
		\left[\begin{NiceMatrix}[columns-width = auto]
			\alpha_1& 1&0&\Cdots&\\
			0& \alpha_3& 1&\Ddots&\\
			\Vdots&\Ddots&\alpha_5&\Ddots&\\
			& &\Ddots &\Ddots &\\&&&&
		\end{NiceMatrix}\right],
	\end{align*}
	with $\alpha_n>0$.
\end{teo}
\begin{proof}
	Take $s\geqslant b$, then $J_s$ has the eigenvalues of its leading principal submatrices $J^{[N]}_s=J^{[N]}+sI_{N+1}$ all positive. The corresponding characteristic polynomials are $P_{N+1}(x-s)= \det\big( xI_{N+1}-J_s^{[N]}\big)$, so that $\det J_s^{[N]}=(-1)^{N+1}P_{N+1}(-s)$ and,  as $-s$ is a lower bound for any possible zero of this monic polynomial, we find  that $(-1)^{N+1}P_{N+1}(-s)>0$. Hence, the leading principal minors of $J_s$ are all positive and the entries on the subdiagonal a superdiagonal are positive. Thus, we conclude, attending to Theorem \ref{teo:Jacobi oscillatory}, that $J_s$ is an oscillatory matrix.
	
	The positive bidiagonal factorization is a consequence of Theorem \ref{teo:LU} applied to $J_s$ for $s>b$.
\end{proof}

The Favard spectral theorem, see \cite{Simon}, ensures for a Jacobi matrix  $J$ the existence of a unique probability measure $\d\psi$ such that
$\int P_n(x) x^m \d\psi(x)=0$, for $m\in\{0,\dots, n-1\}$, that is the characteristic polynomials are orthogonal polynomials, and moreover
$\int x^n\d\psi(x)=(J^n)_{0,0}$. Thus, we see that the tridiagonal matrices to which the classical spectral Favard theorem applies are equivalently described as bounded tridiagonal matrices that after a convenient translation admit a positive bidiagonal factorization. 
\subsection{Mixed multiple orthogonal polynomials on the step-line}

	Mixed multiple orthogonal polynomials were first introduced in 1994 by Sorokin \cite{sorokin1}, and further extended in 1997 by him and van Iseghem in \cite{Sorokin_Van_Iseghem_1} when studying matrix orthogonality of vector polynomials. Ten years later, in 2004, it was rediscovered by Daems and Kuijlaars \cite{Evi_Arno} in the context of multiple non-intersecting Brownian motions, where the name mixed multiple orthogonal was coined. It has been discussed also in \cite{adler, afm, mixto_FUG,Ulises-Sergio-Judith}. Some of the forementioned papers deal  $q\times p$ rectangular matrix of weights of rank 1 at each point of the support. However, the most fitted version for the discussion in this paper are the ones in \cite{Sorokin_Van_Iseghem_1} and \cite{Ulises-Sergio-Judith} in where a $q\times p$ rectangular matrix of functionals or measures, respectively, are considered.

Let us present the mixed multiple orthogonal polynomials, in the step line, as they appears from the $LU$ factorizations of a matrix of moments, following the ideas presented in \cite{afm}. 
\begin{defi}[Matrix of measures]
	Let us consider a matrix of functions, which are right continuous and of bounded variation in a closed interval $\Delta$,
$\Psi=\begin{bNiceMatrix}[small]
		\psi_{1,1}&\Cdots &\psi_{1,p}\\
		\Vdots & & \Vdots\\
		\psi_{q,1}&\Cdots &\psi_{q,p}
		\end{bNiceMatrix}$
and the associated matrix of  Lebesgue--Stieltjes measures
$\d\Psi=\begin{bNiceMatrix}[small]
		\d\psi_{1,1}&\Cdots &\d\psi_{1,p}\\
		\Vdots & & \Vdots\\
		\d\psi_{q,1}&\Cdots &\d\psi_{q,p}
	\end{bNiceMatrix}$.
\end{defi}
\begin{defi}[Monomial matrices]
	Given  $r\in\N$, we consider the semi-infinite matrices of monomials
\begin{align*}
	X_{[r]}\coloneq\begin{bNiceMatrix}
I_r\\
xI_r\\
x^2I_r\\
\Vdots	
\end{bNiceMatrix},
\end{align*}
that, denoting its columns by $X^{(j)}_{[r]}$, we can write 
$X_{[r]}=\begin{bNiceMatrix}[small]
		X^{(1)}_{[r]} &\Cdots & 	X^{(r)}_{[r]} 
	\end{bNiceMatrix}$.
\end{defi}
\begin{defi}[Shift matrices]
	The shift matrix is given by
$	\ell_{[r]}\coloneq \left[\begin{NiceMatrix}[small,columns-width=auto]
		0_r & I_r & 0_r&\Cdots&\\
		0_r & 0_r &I_r &\Ddots&\\
		\Vdots&\Ddots&\Ddots& \Ddots&\\
		&&& &
	\end{NiceMatrix}\right]$.
\end{defi}
\begin{lemma}
Shift matrices act by left multiplication on monomial matrices according to
\begin{align*}
	\ell_{[r]}X_{[r]}&=x X_{[r]}(x), & 	\ell_{[r]}X^{(j)}_{[r]}&=x X^{(j)}_{[r]}(x), & j&\in\{1,\dots,r\}.
\end{align*}
\end{lemma}

\begin{lemma}
	If we denote by 
$	\ell\coloneq\ell_{[1]}=\left[\begin{NiceMatrix}[small,columns-width=auto]
	0 & 1 & 0&\Cdots&\\
	0& 0&1 &\Ddots&\\
	\Vdots&\Ddots&\Ddots& \Ddots&\\
	&&& &
\end{NiceMatrix}\right]$  we have $\ell_{[r]}=\ell^r$.
\end{lemma}

\begin{defi}[Moment matrix]
	The matrix of moments is given by
\begin{align*}
	\mathcal M&\coloneq \int_{\Delta } X_{[q]}(x)
	\d\Psi(x)\big(X_{[p]}(x)\big)^\top=\bigintss_\Delta\,\left[\begin{NiceMatrix}
		\d\Psi(x) & x\d\Psi(x) &x^2\d\Psi(x)&\Cdots\\
		x\d\Psi(x) & x^2\d\Psi(x) &x^3\d\Psi(x)&\Cdots\\
	x^2	\d\Psi(x) & x^3\d\Psi(x) &x^4\d\Psi(x)&\Cdots\\
	\Vdots &\Vdots &\Vdots&
	\end{NiceMatrix}\right]\\&=
\bigint_\Delta\;\,
 \left[\def\arraystretch{1.5}
\begin{NiceArray}{ccc|ccc|w{c}{1pt}W{c}{1cm}c}[]
	\d \psi_{1,1}(x) & \Cdots & \d \psi_{1,p}(x) &x \d \psi_{1,1} (x)& \Cdots&x \d \psi_{1,p} (x)& &&\\
	\Vdots & & \Vdots &\Vdots & & \Vdots &&\Hdotsfor{2}\\
		\d \psi_{q,1}(x) & \Cdots & \d \psi_{q,p}(x) &x \d \psi_{q,1}(x) & \Cdots &x \d \psi_{q,p} (x) & &&\\\hline
	x\d \psi_{1,1}(x) & \Cdots & x\d \psi_{1,p}(x) &x^2 \d \psi_{1,1} & \Cdots &x^2 \d \psi_{1,p} (x) &&& \\
	\Vdots & & \Vdots&\Vdots & & \Vdots& &\Hdotsfor{2}\\
	x\d \psi_{q,1}(x) & \Cdots & x\d \psi_{q,p} (x) &x^2 \d \psi_{q,1} (x) & \Cdots &x^2 \d \psi_{q,p} (x) & & & \\\hline
	 & \Vdotsfor{1} & & & \Vdotsfor{1} & & & &
\end{NiceArray}
\right].
\end{align*}
\end{defi}

\begin{lemma}
	The moment matrices is a structured matrix built up with $q\times p$ blocks $\mathcal M_{n,m}\coloneq\int_{\Delta}x^{n+m}\DJ\d\Psi(x)\in\R^{q\times p}$. In fact, is Hankel by blocks, i.e. $\mathcal M_{n+1,m}=	\mathcal M_{n,m+1}$. 
\end{lemma}

This fact can be reformulated as follows:

\begin{pro}
	The moment matrix satisfies
\begin{align}\label{eq:symmetry_mixed}
	\ell_{[q]} 	\mathcal M=	\mathcal M(\ell_{[p]})^\top.
\end{align}
\end{pro}
\begin{proof}
	It follows from
\begin{align*}
	\ell_{[q]} 	\mathcal M&=\int_{\Delta } \ell_{[q]}X_{[q]}(x)\d\Psi(x)(X_{[p]}(x))^\top=\int_{\Delta } x X_{[q]}(x)\d\Psi(x)(X_{[p]}(x))^\top\\&=\int_{\Delta } X_{[q]}(x)\d\Psi(x)(X_{[p]}(x))^\top(\ell_{[p]})^\top.
\end{align*}
\end{proof}

Now, let us assume that the moment matrix $	\mathcal M$ has a Gauss--Borel factorization, i.e. $	\mathcal M=L^{-1}U^{-1}$, where $L$ is lower triangular and $U$ upper triangular, such that none of the diagonal entries in both triangular matrices are zero, and the respective inverse matrices makes sense. It is well known that such $LU$ factorizations do exist whenever the leading principal submatrices are nonsingular.
\begin{defi}[Matrix polynomials]
	Let us consider the matrices
\begin{align*}
	B&\coloneq L X_{[q]}=\left[\begin{NiceMatrix}
		\mathscr B_0\\ \mathscr B_1\\\Vdots
	\end{NiceMatrix}\right],&
	A&\coloneq (X_{[p]})^\top U=\left[\begin{NiceMatrix}
		\mathscr A_0 & \mathscr A_1 & \Cdots
	\end{NiceMatrix}\right],
\end{align*}
where $\mathscr B_n(x)$ is a $q\times q$ matrix polynomial and $\mathscr A_n(x)$ is a $p\times p$ matrix polynomial.
Observe that
\begin{align*}
B&=\begin{bNiceMatrix}
		B^{(1)} &\Cdots &B^{(q)}
\end{bNiceMatrix},& B^{(b)}&=L X_{[q]}^{(b)}, & b&\in\{1,\dots,q\},\\
A&=\begin{bNiceMatrix}
	A^{(1)} \\\Vdots \\A^{(p)}
\end{bNiceMatrix}, & A^{(a)}&= (X_{[p]}^{(a)})^\top U, & a&\in\{1,\dots,p\}.
\end{align*}
\end{defi}
\begin{lemma}\label{lemma:blocks_mixed_multiple_orthogonality}
	The semi-infinite vectors
\begin{align*}
	B^{(b)}&=\begin{bNiceMatrix}
		B^{(b)}_0\\[2pt]
		B^{(b)}_1\\
		\Vdots
	\end{bNiceMatrix}, & 	A^{(a)}=\left[\begin{NiceMatrix}
	A^{(a)}_0&
	A^{(a)}_1 &
	\Cdots
\end{NiceMatrix}\right], 
\end{align*}
have as entries polynomials with degrees
\begin{align}\label{eq:degrees}
\deg B^{(b)}_{n} &=\left\lceil\frac{n+2-b}{q} \right\rceil-1 , 	&\deg A^{(a)}_{n} =\left\lceil\frac{n+2-a}{p} \right\rceil-1.
\end{align}
For the block polynomials we have
\begin{align*}
\mathscr	B_n&=\begin{bNiceMatrix}[columns-width=.5cm]
		B^{(1)}_{nq}&\Cdots & B^{(q)}_{nq}\\
		\Vdots & &\Vdots\\
		B^{(1)}_{nq+q-1}&\Cdots&B^{(q)}_{nq+q-1}
	\end{bNiceMatrix}, & 
\mathscr	A_n&=\begin{bNiceMatrix}
	A^{(1)}_{np}&\Cdots&& A^{(1)}_{np+p-1}\\
	\Vdots & &&\Vdots\\\\
A^{(p)}_{np}&\Cdots &&A^{(p)}_{np+p-1}
\end{bNiceMatrix}.
\end{align*}
\end{lemma}

\begin{lemma}\label{lemma:degrees}
	\begin{enumerate}
		\item For $r,n\in\N$, $n\geqslant r$, we have
		\begin{align}\label{eq:ceil}
			\sum_{a=1}^r	\left\lceil\frac{n+1-a}{r} \right\rceil = n.
		\end{align}
	\item For the degrees we find $
		\sum_{b=1}^q (\deg B^{(b)}_{n}+1) =\sum_{a=1}^p (\deg A^{(a)}_{n}+1)= n+1$.
	\end{enumerate}
\end{lemma}
\begin{proof}
\begin{enumerate}
	\item 	Let us consider $n = kr +j$, where $j = 0, \dots r-1$.
For $a=1, \dots, j$, we have that
\begin{align*}
	\left\lceil\frac{n+1-a}{r} \right\rceil = \left\lceil\frac{kr +j+1-a}{r} \right\rceil = k + \left\lceil\frac{ j+1-a}{r} \right\rceil 
\end{align*}
and we get $n_a = k + 1$.
For $a=j+1, \dots, r$, we find
\begin{align*}
	\left\lceil\frac{n+1-a}{r} \right\rceil = \left\lceil\frac{kr +j+1-a}{r} \right\rceil = \left\lceil\frac{(k-1)r+ r-a+j+1)}{r} \right\rceil= (k-1) + \left\lceil\frac{ r-a+j+1)}{r} \right\rceil 
\end{align*}
and, consequently, $n_a = k$. Therefore, Equation \eqref{eq:ceil} follows.

\item It follows from the previous result and \eqref{eq:degrees}.
\end{enumerate}
\end{proof}

\begin{pro}[Biorthogonality]
	The following biorthogonality holds
	\begin{align*}
		\int_\Delta B(x)\d\Psi (x) A(x)&=I, 
	\end{align*}
that entrywise can be written
\begin{align}\label{eq:mixed_multiple_biorthogonality}
	\sum_{b=1}^q\sum_{a=1}^p\int_\Delta B^{(b)}_n(x)\d\psi_{b,a} (x) A^{(a)}_m(x)&=\delta_{n,m}, & n,m\in\N_0.
\end{align}
\end{pro}
\begin{proof}
	From the definition of the moment matrix $\mathcal M$ and the Gauss--Borel factorization $\mathcal M=L^{-1}U^{-1}$ we get
	\begin{align*}
		\int_{\Delta} X_{[q]}(x)\d\Psi (x)\big(X_{[p]}(x)\big)^\top=L^{-1} U^{-1}.
		\end{align*}
	By left and right multiplication by the triangular factors $L$ and $U$, respectively, we get
		\begin{align*}
		\int_{\Delta} L X_{[q]}(x)\d\Psi (x)\big(X_{[p]}(x)\big)^\top U=I
	\end{align*}
and recalling the definition of $B$ and $A$ we deduce that
$\int_{\Delta}B(x)\d\Psi (x)A(x)=I$,
and the result follows immediately.
\end{proof}
\begin{rem}[Matrix biorthogonality]
\begin{enumerate}
	\item 	If $p=q$, we recover the well-known matrix biorthogonality, see \cite{damanik,sinap,Duran_Grunbaum} for matrix orthogonality,
		\begin{align*}
		\int_\Delta \mathscr B_n(x)\d\Psi (x) \mathscr A_m(x)&=\delta_{n,m}I_p, & n,m\in\N_0.
	\end{align*}
\item For $p\neq q$, this matrix orthogonality is lost. However, if we denote for $r\in\N$ by $\mathscr B^{[r]}_n$ the $r\times q$ matrix of polynomials obtained from the $\infty\times q$ matrix $B$ by taking consecutive submatrices of size $r\times q$, and similarly for $\mathscr A^{[r]}_n$, i.e., obtained from
the $p\times \infty$ matrix $A$ consecutive submatrices of size $p\times r$, we get the following generalized matrix biorthogonality
\begin{align*}
	\int_{\Delta } \mathscr B_n^{[r]}\d\Psi(x)\mathscr A_m^{[r]}(x)&=\delta_{n,m} I_r , & n,m&\in\N_0.
\end{align*}
\end{enumerate}
\end{rem}

From biorthogonality \eqref{eq:mixed_multiple_biorthogonality} we get mixed multiple orthogonal relations as follows:
\begin{coro}[Mixed multiple orthogonality]\label{coro:mixed_multiple_orthogonality}
	The following orthogonality relations 
\begin{align*}
		\sum_{a=1}^p \int_\Delta A^{(a)}_{n} (x)\d \psi_{b,a} (x)x^m &=0, &k &\in\left\{ 0 ,\dots, \deg B^{(b)}_{n-1}\right\}, & b&\in\{1,\dots,q\},\\
		\sum_{b=1}^q 	\int _\Delta B^{(b)}_{n} (x)\d \psi_{b,a}(x) x^m &=0, &k &\in\left\{0 ,\dots, \deg A^{(a)}_{n-1}\right\},& a&\in\{1,\dots,p\}.
\end{align*}
are satisfied.
\end{coro}

\begin{rem}
	For $q=1$ the mixed multiple orthogonality is the well known multiple orthogonality, or $p$-orthogonality, with $A^{(a)}_n$ the type I multiple orthogonal polynomials and the $B_n$ the type II multiple orthogonal polynomials, see \cite{nikishin_sorokin,Ismail}.
\end{rem}

We now discuss the connection of the Gauss--Borel factorization  of the moment matrix  and the Cauchy transforms. 
\begin{defi}
	Let us consider the formal semi-infinite matrices 
\begin{align}\label{eq:CTGB}
	C(z)&\coloneq z^{-1}\big(X_{[q]}(z^{-1})\big)^\top L^{-1}=\left[\begin{NiceMatrix}
		\mathscr C_0&\mathscr C_1 &\Cdots
	\end{NiceMatrix}\right], &
	 D(z)&\coloneq z^{-1}U^{-1} X_{[p]}(z^{-1})=\begin{bNiceMatrix}
	 	\mathscr D_0\\ \mathscr D_1 \\\Vdots
	 \end{bNiceMatrix}.
\end{align}
With  $\mathscr C_n,\mathscr D_n$ being $q\times p$  rectangular blocks.
\end{defi}
\begin{rem}
The previous definition is formal, as the product of matrices involves series instead of sums, hence to have a meaning we must ensure the convergence of these series. 
\end{rem}
\begin{rem}
	We have that
\begin{align*}
	C&=\begin{bNiceMatrix}
		C^{(1)}\\
		\Vdots\\
		C^{(q)}
	\end{bNiceMatrix}, & D&=\begin{bNiceMatrix}
	D^{(1)} & \Cdots & D^{(p)}
\end{bNiceMatrix}
\end{align*}
where $C^{(b)}=\begin{bNiceMatrix}[small]
	C^{(b)}_0 &	C^{(b)}_1&\Cdots 
\end{bNiceMatrix}$ are semi-infinite row vectors and $D^{(b)}=\begin{bNiceMatrix}[small]
D^{(b)}_0 \\D^{(b)}_1\\\Vdots 
\end{bNiceMatrix}$ semi-infinite column vectors.
The block matrices are
\begin{align*}
	\mathscr	C_n&=\begin{bNiceMatrix}
		C^{(1)}_{np}&\Cdots && C^{(1)}_{np+p-1}\\
		\Vdots & &&\Vdots\\
		C^{(q)}_{np}&\Cdots&&C^{(q)}_{np+p-1}
	\end{bNiceMatrix}, & 
	\mathscr	D_n&=\begin{bNiceMatrix}
		D^{(1)}_{nq}&\Cdots&& D^{(p)}_{nq}\\
		\Vdots & &&\Vdots\\
		D^{(1)}_{nq+q-1}&\Cdots &&D^{(p)}_{nq}
	\end{bNiceMatrix}.
\end{align*}
\end{rem}
\begin{pro}[Cauchy transforms]
For $|x|<|z|$,  matrices in \eqref{eq:CTGB} are the following Cauchy transforms
\begin{align*}
	C(z)&= \int \frac{\d\Psi(x)}{z-x}A(x), & D(z)&=\int B(x)\frac{\d\Psi(x)}{z-x}.
\end{align*} 
\end{pro}

\begin{proof}
	We have
\begin{align*}
		C(z)= z^{-1}\big(X_{[q]}(z^{-1})\big)^\top 	\mathcal M U=z^{-1}X^\top_{[q]}(z^{-1})\int X_{[q]}(x)\d\Psi(x)(X_{[p]}(x))^\top U.
\end{align*}
Now, notice that for $|x|<|z|$ it holds that
\begin{align*}
	z^{-1}\big(X_{[r]}(z^{-1})\big)^\top X_{[r]}(x)=\frac{1}{z}\sum_{n=0}^\infty\frac{x^n}{z^n} I_r=\frac{I_r}{z-x},
\end{align*}
so that, recalling that $(X_{[p]}(x))^\top U=A(x)$, we get
$	C(z)=\int \frac{\d\Psi(x)}{z-x}A(x)$.
Analogously, 
 \begin{align*}
 	D(z)= L	\mathcal M X_{[p]}(z^{-1})z^{-1} =\int LX_{[q]}(x)\d\Psi(x)(X_{[p]}(x))^\top X_{[p]}(z^{-1})z^{-1}=\int B(x)\frac{\d\Psi(x)}{z-x}.
 \end{align*}
\end{proof}

\begin{rem}
	Entrywise, we find
	\begin{align*}
		C^{(b)}_n(z)&=\sum_{a=1}^p\int \frac{\d\psi_{b,a}(x)}{z-x}A_n^{(a)}(x), &	D^{(a)}_n(z)&=\sum_{b=1}^q\int B_n^{(b)}(x)\frac{\d\psi_{b,a}(x)}{z-x},
	\end{align*}
and block entrywise
	\begin{align*}
	\mathscr C_n(z)&=\int \frac{\d\Psi(x)}{z-x}\mathscr A_n(x), &	\mathscr D_n(z)&=\int \mathscr B_n(x)\frac{\d\Psi(x)}{z-x}.
\end{align*}
\end{rem}

Now, let us discuss how these polynomials connect with the matrix Hermite--Padé problem as considered in \cite{Sorokin_Van_Iseghem_3}.
For that aim, we first introduce:
\begin{defi}[Stieltjes--Markov functions]
	Let us consider the  Stieltjes--Markov functions given by
\begin{align*}
	\hat{\psi}_{b,a}(z) &\coloneq 	\int _\Delta \frac{\d \psi_{b,a}(x) }{z-x},
\end{align*}
i.e., the Cauchy transforms of the measures. We also introduce  two families of polynomials of the second kind linked to the orthogonal polynomials:
\begin{align*}
R^{(a)}_{n}(z) &\coloneq \sum_{b=1}^q \int _\Delta \frac{B^{(b)}_{n} (z)-B^{(b)}_{n} (x) }{z-x}\d \psi_{b,a}(x) , &
Q^{(b)}_{n}(z) &\coloneq \sum_{a=1}^p \int _\Delta \frac{A^{(a)}_{n} (z)-A^{(a)}_{n} (x) }{z-x}\d \psi_{b,a}(x) . 
\end{align*}
\end{defi}

\begin{pro}[Matrix Hermite--Padé]\label{pro:matrix_Hermite_Padé}
	\begin{enumerate}
		\item 	The following simultaneous approximation holds 
\begin{align*}
	\sum_{b=1}^q 	 B^{(b)}_{n} (z) \hat{\psi}_{b,a}(z) &=R^{(a)}_{n}(z) + O\left(\frac{1}{z^{n_a+1}}\right),& z&\to\infty,
\end{align*}
with
\begin{align*}
	n_a \coloneq \deg A^{(a)}_{n-1} +1 = \left\lceil\frac{n+1-a}{p} \right\rceil.
\end{align*}
We have
$\sum_{a=1}^p n_a=n$ and $\sum_{b=1}^q (\deg B^{(b)}_{n}+1)= n+1$.
\item Analogously,
the simultaneous approximation 
\begin{align*}
	\sum_{a=1}^p 	\hat{\psi}_{b,a}(z) A^{(a)}_{n} (z) &=Q^{(b)}_{n}(z) + O\left(\frac{1}{z^{m_b+1}}\right),& z&\to\infty,
\end{align*}
with
\begin{align*}
	m_b\coloneq \deg B^{(b)}_{n-1} +1 = \left\lceil\frac{n+1-b}{q} \right\rceil,
\end{align*}
is satisfied.
We have $\sum_{b=1}^q m_b=n$ and $\sum_{a=1}^p (\deg A^{(a)}_{n}+1)= n+1$.
	\end{enumerate}
\end{pro}

\begin{proof}
Let us check only the first case. The other follows by similar arguments.
Observe that
\begin{align*}
\sum_{b=1}^q 	 B^{(b)}_{n} (z) \hat{\psi}_{b,a}(z) &= 
\sum_{b=1}^q 	 B^{(b)}_{n} (z) \int _\Delta \frac{\d \psi_{b,a}(x) }{z-x}\\
&= \sum_{b=1}^q 	 \int _\Delta \frac{B^{(b)}_{n} (z)-B^{(b)}_{n} (x) }{z-x}\d \psi_{b,a}(x) + \sum_{b=1}^q \int _\Delta \frac{B^{(b)}_{n} (x) }{z-x}\d \psi_{b,a}(x)
\\&
=R^{(a)}_n(x)+ \sum_{k=0}^{+\infty} \frac{1}{z^{k+1}} \sum_{b=1}^q\int _\Delta B^{(b)}_{n} (x) x^k\d \psi_{b,a}(x).
\end{align*}
Using now the mixed multiple orthogonality conditions, see Corollary \ref{coro:mixed_multiple_orthogonality}, we get the result. The degrees follow from Lemma \ref{lemma:degrees}. This is exactly the matrix Hermite--Padé problem that appears in \cite{Sorokin_Van_Iseghem_3}.
\end{proof}

Finally, we discuss the recursion relations
\begin{pro}[Banded recursion matrix]
\begin{enumerate}
	\item The following relation is fulfilled
\begin{align*}
	L\ell_{[q]} L^{-1}=U^{-1}(\ell_{[p]}  )^\top U.
\end{align*}
\item The semi-infinite matrix $T\coloneq L\ell_{[q]}  L^{-1}=U^{-1}(\ell_{[p]} ) ^\top U$ is a banded matrix with $p$ subdiagonals, $q$ superdiagonals, and main diagonal possibly with nonzero entries. Additionally, the extreme diagonal entries are nonzero.
\item The following recursion relations hold
\begin{align*}
	T B&=x B, & A T=xA.
\end{align*}
\end{enumerate}
\end{pro}
\begin{proof}
	\begin{enumerate}
		\item From the Gauss--Borel factorization and \eqref{eq:symmetry_mixed} we get
		\begin{align*}
		\ell_{[q]}  L^{-1} U^{-1}=L^{-1} U^{-1}(\ell_{[p]} )^\top,
		\end{align*}
	so that
		\begin{align*}
	L	\ell_{[q]}  L^{-1} =U^{-1}(\ell_{[p]} )^\top U.
	\end{align*}
\item The matrix $L\ell_{[q]}  L^{-1}$ has all its superdiagonals above the the first $q$-th superdiagonal with zero entries, while $U^{-1}(\ell_{[p]} )^\top U$ has all its subdiagonals below the first $p$-th subdiagonal with zero entries. Hence, $T$ is a general banded matrix with $p+q+1$ diagonals possibly with nonzero entries.
\item By definition $B=L X_{[q]}$ so that $TB=L\ell_{[q]}  L^{-1}L X_{[q]}=L\ell_{[q]}  X_{[q]}=xB$. Similarly, also by definition, we have $A =(X_{[p]})^\top U $ so that 
$A T= (X_{[p]})^\top U U^{-1}(\ell_{[p]}  )^\top U=(X_{[p]})^\top(\ell_{[p]}  )^\top U=x A$.
	\end{enumerate}
\end{proof}

This banded recursion matrix is the object of study of this paper. It will be the departure point in the next sections. We have considered a matrix of measures and the associated matrix of moments and derived the mixed multiple orthogonality as well as the banded recursion matrix. 
The aim in this paper is to get conditions on the banded matrix so that we can go back this way, to retrieve the matrix of positive measures and the mixed multiple orthogonal polynomials from the recursion matrix; i.e., to get a spectral Favard theorem.

\section{Recursion polynomials and the characteristic polynomial}

We begin by introducing the recursion polynomials associated to the banded matrix $T$, with truncations given in \eqref{eq:monic_Hessenberg},  as the entries of semi-infinite left and right eigenvectors:

\begin{defi}[Left and right recursion polynomials]
	Associated with the semi-infinite banded matrix $T$ we consider the semi-infinite vectors
\begin{align*}
	A^{(a)}&=\left[\begin{NiceMatrix}
		A^{(a)}_0 & 	A^{(a)}_1& \Cdots
	\end{NiceMatrix}\right], & a&\in\{1,\dots, p\},&	B^{(b)}&=\begin{bNiceMatrix}
	B^{(b)}_0 \\[2pt]	B^{(b)}_1\\ \Vdots
\end{bNiceMatrix}, & b&\in\{1,\dots, q\},
\end{align*}
 that are left and right eigenvectors with eigenvalue $x$ of $T$, i.e.
 \begin{align*}
 	A^{(a)}T&=x 	A^{(a)},& a&\in\{1,\dots, p\},& TB^{(b)}&=xB^{(b)}, &b&\in\{1,\dots, q\}.
 \end{align*}

The entries of these left and right eigenvectors are polynomials in the eigenvalue $x$, known as left and right recursion polynomials, respectively, determined by the initial conditions
\begin{align}
	\label{eq:initcondtypeI}
	\begin{cases}
		A^{(1)}_0=1 , \\
		A^{(1)}_1= \nu^{(1)}_1 , \\
		\hspace{.895cm} \vdots \\
		A^{(1)}_{p-1}=\nu^{(1)}_{p-1} ,
	\end{cases}
	&&
	\begin{cases}
		A^{(2)}_0=0 , \\
		A^{(2)}_1= 1 , \\
		A^{(2)}_2= \nu^{(2)}_2 , \\
		\hspace{.895cm} \vdots \\
		A^{(2)}_{p-1}=\nu^{(2)}_{p-1} ,
	\end{cases}
	&& \cdots &&
	\begin{cases}
		A^{(p)}_0 =0 , \\
		\hspace{.915cm} \vdots \\
		A^{(p)}_{p-2} = 0 , \\
		A^{(p)}_{p-1} = 1,
	\end{cases}
\end{align}
with $\nu^{(i)}_{j}$ being arbitrary constants, and 
\begin{align}
	\label{eq:initcondtypeII}
	\begin{cases}
		B^{(1)}_0=1 , \\
		B^{(1)}_1= \xi^{(1)}_1 , \\
		\hspace{.895cm} \vdots \\
		B^{(1)}_{q-1}=\xi^{(1)}_{q-1} ,
	\end{cases}
	&&
	\begin{cases}
		B^{(2)}_0=0 , \\
		B^{(2)}_1= 1 , \\
		B^{(2)}_2= \xi^{(2)}_2 , \\
		\hspace{.895cm} \vdots \\
		B^{(2)}_{q-1}=\xi^{(2)}_{q-1} ,
	\end{cases}
	&& \cdots &&
	\begin{cases}
		B^{(p)}_0 =0 , \\
		\hspace{.915cm} \vdots \\
		B^{(p)}_{q-2} = 0 , \\
		B^{(p)}_{q-1} = 1,
	\end{cases}
\end{align}
with $\xi^{(i)}_{j}$ also being arbitrary, respectively. We also define the initial condition matrices
\begin{align}
	\label{eq:ic}
	\nu&\coloneq \begin{bNiceMatrix}
		1& 0 & \Cdots& && 0 \\
		\nu^{(1)}_1 & 1 & \Ddots&& & \Vdots \\
		\Vdots & \Ddots & \Ddots& && \\
		&&&& &\\&&&&&0\\
		\nu^{(1)}_{p-1} &\Cdots& && \nu^{(p-1)}_{p-1}& 1 
	\end{bNiceMatrix} ,&
\xi&\coloneq \begin{bNiceMatrix}
	1& 0 & \Cdots& && 0 \\
	\xi^{(1)}_1 & 1 & \Ddots&& & \Vdots \\
	\Vdots & \Ddots & \Ddots& && \\
	&&&& &\\&&&&&0\\
	\xi^{(1)}_{q-1} &\Cdots& && \xi^{(q-1)}_{q-1}& 1 
\end{bNiceMatrix}.
\end{align}
\end{defi}
Once the initial conditions are fixed, the remaining polynomials are found by:

\begin{pro}[General recursion relations]
The recursion 	polynomials are uniquely determined by the initial conditions \eqref{eq:initcondtypeI} and \eqref{eq:initcondtypeII} and the recursion relations
\begin{align}\label{eq:recursion_dual_A}
A^{(a)}_{n-q} T_{n-q,n}+ \cdots +A^{(a)}_{n+p} T_{n+p,n}&= x A^{(a)}_{n}, & n &\in\{0,1,\ldots\}, & a &\in\{1,\dots, p\}, & A_{-q}^{(a)} &=\dots=A^{(a)}_{-1}= 0,\\
\label{eq:recursion_B}
T_{n,n-p}B^{(b)}_{n-p} + \cdots + T_{n,n+q}B^{(b)}_{n+q} &= x B^{(b)}_{n}, & n &\in\{0,1,\ldots\}, & b&\in\{1,\dots, q\}, & B_{-p}^{(b)} &=\dots=B^{(b)}_{-1}= 0.
\end{align}
\end{pro}

We use the ceiling function $\lceil x\rceil$ that maps $x$ to the least integer greater than or equal to $x$.

\begin{pro}\label{pro:degrees}
	For the degrees of the left and right recursion polynomials we find
	\begin{align*}
	\deg A^{(a)}_n&=\left	\lceil \frac{n+2-a}{p}\right \rceil -1, & \deg B^{(b)}_n&=\left	\lceil \frac{n+2-b}{q}\right \rceil -1.
	\end{align*}
\end{pro}
\begin{proof}
	By inspection we can check that, for $j\in\{1,\dots, p\}$ and $k\in\N_0$, it holds that $\deg A^{(a)}_{kp+j} = k$, for $a\in \{1, \dots, j+1\}$ and 
$\deg A^{(a)}_{kp+j}=k-1$ for $a\in \{j+2, \dots , p\}$ and that, for $j\in\{1,\dots,q\}$ and $k\in\N_0$, $\deg B^{(b)}_{kq+j} = k$, for $b\in \{1, \dots, j+1\}$ and $\deg B^{(b)}_{kp+j}=k-1$ for $b =\{ j+2, \dots , q\}$. 

However, we notice that
\begin{align*}
\left	\lceil \frac{n+2-a}{p}\right \rceil -1= \left	\lceil \frac{kp+j+2-a}{p}\right \rceil-1 =k-1+ \left	\lceil \frac{j+2-a}{p}\right \rceil 
\end{align*}
but
\begin{align*}
	\left	\lceil \frac{j+2-a}{p}\right \rceil =
	\begin{cases}
		1, & a\in \{1, \dots, j+1\},\\
0, & a\in \{j+2, \dots , p\},
	\end{cases}
\end{align*}
and the stated result follows. For the recursion polynomials $B^{(b)}_n$ we proceed analogously. 
\end{proof}


\begin{defi}[Characteristic polynomials]
	For the semi-infinite matrix $T$ we consider the polynomials $P_N(x)$ as the characteristic polynomials of the truncated matrices $T^{[N-1]}$, i.e.,
\begin{align*}
	P_{N}(x)&\coloneq\begin{cases}
		1, & N=0,\\\det\big(xI_N-T^{[N-1]}\big), & N\in\N.
	\end{cases}
\end{align*}
\end{defi}
Obviously, $\deg P_N=N$.
For Hessenberg matrices \cite{PBF_1} it happens that the characteristic polynomials up to a factor coincides with the right recursion polynomials. However, for the banded situation this does not hold in general. Nevertheless, there is a relation between determinants of the recursion polynomials, right or left, with the characteristic polynomials of the banded matrix $T$. Let us see this.

\begin{defi}	Let us introduce the following matrices of left and right recursion polynomials
\begin{align*}
A_N &\coloneq	\begin{bNiceMatrix}
		A^{(1)}_N& \Cdots & A^{(1)}_{N+p-1}	 \\[2pt]
		\Vdots & & \Vdots \\[2pt]
A^{(p)}_N	& \Cdots & A^{(p)}_{N+p-1}
	\end{bNiceMatrix}, & B_N &\coloneq	\begin{bNiceMatrix}
	B^{(1)}_N & \Cdots & B^{(q)}_N \\[2pt]
	\Vdots & & \Vdots \\[2pt]
	B^{(1)}_{N+q-1} & \Cdots & B^{(q)}_{N+q-1}
\end{bNiceMatrix},& N&\in \N_0 ,
\end{align*}
and the following products
\begin{align*}
	\alpha_N &\coloneq (-1)^{(p-1)N}T_{p,0}\cdots T_{N+p-1,N-1}, &\beta_N &\coloneq (-1)^{(q-1)N}T_{0,q}\cdots T_{N-1,N+q-1} , & N&\in \N,
\end{align*}
and $\alpha_0=\beta_0=1$.
\end{defi}

\begin{rem}
	These are inspired by the matrix  polynomials blocks  given in the Gauss--Borel construction of mixed multiple orthogonality, see Lemma \ref{lemma:blocks_mixed_multiple_orthogonality}. In fact, for $M\in\N_0$, $A_{Mp}=\mathscr A_M$ and $B_{Mq}=\mathscr B_{M}$.
\end{rem}

Recall that as the entries in the extreme diagonals do not cancel
\eqref{eq:not_zero_extreme_diagonals} we have that
$	\alpha_N,\beta_N\neq 0$.
In terms of these objects we found the following important result:
\begin{teo}\label{teo:determinants_mixed}
For $N\in\N_0$, the characteristic polynomials 	and determinants of left and right recursion polynomial blocks satisfy
	\begin{align*}
	P_N(x)&=\alpha_N 	\det A_N(x)=\beta_N 	\det B_N(x).
	\end{align*}
\end{teo}

\begin{proof}
For $N=0$ we have that $\det A_0=\det \nu=1$. For $N=1$ we get 
\begin{align*}
	T_{p,0}\det A_1&=	\begin{vNiceMatrix}[]
		A^{(1)}_1& \Cdots & A^{(1)}_{p-1}&T_{p,0}A^{(1)}_{p}\\
		\Vdots & & \Vdots & \Vdots \\
		A^{(p)}_{1}	 & \Cdots & A^{(p)}_{p-1}&T_{p,0}A^{(p)}_{p}
	\end{vNiceMatrix} =	\begin{vNiceMatrix}[]
	A^{(1)}_1& \Cdots & A^{(1)}_{p-1}&(x-T_{0,0})A^{(1)}_{0}\\
	\Vdots & & \Vdots & \Vdots \\
	A^{(p)}_{1}	 & \Cdots & A^{(p)}_{p-1}&(x-T_{0,0})A^{(p)}_{0}
\end{vNiceMatrix} 
\end{align*}
where we have used the recursion \eqref{eq:recursion_dual_A} in the last column of this determinant.
Now we express this last determinant as the following product of determinants

\begin{align*}
\begin{vNiceMatrix}[]
	A^{(1)}_1& \Cdots & A^{(1)}_{p-1}&(x-T_{0,0})A^{(1)}_{0}\\
	\Vdots & & \Vdots & \Vdots \\
	A^{(1)}_{p}	 & \Cdots & A^{(p)}_{p-1}&(x-T_{0,0})A^{(p)}_{0}
\end{vNiceMatrix} 	&= 
		\begin{vNiceMatrix}[]
		A^{(1)}_0& \Cdots & A^{(1)}_{p-1}\\
		\Vdots & & \Vdots \\
		A^{(p)}_{0}	 & \Cdots & A^{(p)}_{p-1}
	\end{vNiceMatrix} 	\left|\begin{NiceArray}{cccc|c}
			0&\Cdots&&0&	x-T_{0,0} \\\hline
		\Block{4-4}<\large>{I_{p-1}}&&&&0\\
		&&&&\\
		&&&&\Vdots\\
		&&&&0
	\end{NiceArray}\right|
\\&= (-1)^{p+1}(x-T_{0,0}).
\end{align*}
We proceed similarly up to $N=p-1$, so for $N\in\{2, \dots, p-1\}$, we get that
\begin{multline*}
	T_{p,0} T_{p+1,1}\dots T_{N+p-1,N-1} \det A_N\coloneq	\begin{vNiceMatrix}
		A^{(1)}_N & \Cdots & A^{(1)}_{p-1} &T_{p,0}  A^{(1)}_p & \Cdots &T_{N+p-1,N-1} A^{(1)}_{N+p-1}	\\
		\Vdots & & \Vdots &\Vdots & & \Vdots\\
		A^{(p)}_N& \Cdots & A^{(p)}_{p-1} & T_{p,0}A^{(p)}_p 
		& \Cdots & T_{N+p-1,N-1} A^{(p)}_{N+p-1}	\\
	\end{vNiceMatrix} 
	\\=
	\begin{vNiceMatrix}[small]
		A^{(1)}_N & \Cdots & A^{(1)}_{p-1} &(x-T_{0,0}) A^{(1)}_{0} -T_{1,0} A^{(1)}_{1}-\dots -T_{N-1,0} A^{(1)}_{N-1} & \dots &-T_{0,N-1} A^{(1)}_{0} -T_{1,N-1} A^{(1)}_{1}-\dots+(x-T_{N-1,N-1} )A^{(1)}_{N-1}	\\
		\Vdots & & \Vdots & \Vdots & & \Vdots\\
		A^{(p)}_N& \Cdots & A^{(p)}_{p-1} & (x-T_{0,0}) A^{(p)}_{0} -T_{1,0} A^{(p)}_{1}-\dots -T_{N-1,0} A^{(p)}_{N-1} 
		& \dots & -T_{0,N-1} A^{(p)}_{0} -T_{1,N-1} A^{(p)}_{1}-\dots+(x-T_{N-1,N-1} )A^{(p)}_{N-1}	\\
	\end{vNiceMatrix} 
	\\	= \setlength\extrarowheight{5pt}
	\begin{vNiceMatrix}[]
		A^{(1)}_0& \Cdots & A^{(1)}_{p-1}\\
		\Vdots & & \Vdots \\
		A^{(p)}_{0}	 & \Cdots & A^{(p)}_{p-1}
	\end{vNiceMatrix} \left|\begin{NiceArray}{wc{2.7cm}cc|wc{4.4cm}ccc}
		\Block[c]{4-3}<\Large>{0_{N\times (p-N)}}&&&\Block[c]{4-4}<\Large>{x I_N-T^{[N-1]}}&&&\\
		&&&&&&\\
		&&&&&&\\
		&&&&&&\\
		\hline 
		\Block[c]{3-3}<\Large>{I_{p-N}}	&&&\Block[c]{3-4}<\Large>{0_{(p-N) \times N}}&&&\\
		&&&&&&\\
		&&&&&&
	\end{NiceArray}\right|
	= (-1)^{N(p-N)} P_N(x),
\end{multline*}
where, in the second equality, we have used the recursion relation \eqref{eq:recursion_dual_A} in the last $N$ columns and cancel the contributions already present in the previous columns.

For $N\geqslant p$, using the recursion relation similarly as above we get
\begin{align*}
	T_{N,N-p}\cdots T_{N+p-1,N-1}\det A_N&=\det M, &M&\coloneq
\begin{bNiceMatrix}
	A^{(1)}_0 &\Cdots &&&A^{(1)}_{N-1}\\
	\Vdots & &&&\Vdots\\
		A^{(p)}_0 &\Cdots &&&A^{(p)}_{N-1}
\end{bNiceMatrix}
\begin{bNiceMatrix}
	-T_{0,N-p}&\Cdots &-T_{0,N-1}\\
	\Vdots & &\Vdots\\
	-T_{N-p-1,N-p} &\Cdots& -T_{N-p-1,N-1}\\
x-T_{N-p, N-p}&\Cdots&-T_{N-p,N-1}\\
\Vdots&\Ddots&\Vdots\\
T_{N-1, N-p}&\Cdots&	x-T_{N-1,N-1}
\end{bNiceMatrix}.
\end{align*}
In order to compute this determinant we notice that

\begin{align*}
\setlength\extrarowheight{5pt}
\left[\begin{NiceArray}{ccc|cccc}
\Block{	4-3}<\Large>{0_{(N-p)\times p}}&&&\Block{4-4}<\Large>{I_{N-p}}&&&\\
	&&&&&&\\
	&&&&&&\\
	&&&&&&\\
	\hline 
	A^{(1)}_0&\Cdots&A_{p-1}^{(1)}&A_p^{(1)}&\Cdots&&A^{(1)}_{N-1}\\
	\Vdots&&\Vdots&\Vdots&&&\Vdots\\
	&&&&&&\\
A^{(p)}_0&\Cdots&A_{p-1}^{(p)}&A_p^{(p)}&\Cdots&&A^{(p)}_{N-1}
\end{NiceArray}\right](xI_N-T^{[N-1]})=
\setlength\extrarowheight{0pt}\left[\begin{NiceArray}{ccWc{1cm}c|Wc{1.5cm}cc}
-T_{p,0}&\Cdots&&-T_{p,N-p-1}&\Block{4-3}<\Large>{C_{(N-p)\times p}}&&\\
	0&\Ddots&&\Vdots&&&\\\Vdots	&\Ddots&&&&&\\[16pt]	0&\Cdots&0&-T_{N-1,N-p-1}&&&\\\hline	\Block{3-4}<\Large>{0_{p\times (n-p)}}&&&&\Block{3-3}<\Large>{M}&&\\	&&&&&&\\	&&&&&&
\end{NiceArray}\right]
\end{align*}
where $C$ is an $(N-p)\times p$ submatrix of $xI_N- T^{[N-1]}$ that is not relevant for the reasoning. Observe that
\begin{align*}\setlength\extrarowheight{5pt}
\left|\begin{NiceArray}{ccc|cccc}
	\Block{	4-3}<\Large>{0_{(N-p)\times p}}&&&\Block{4-4}<\Large>{I_{N-p}}&&&\\
	&&&&&&\\
	&&&&&&\\
	&&&&&&\\
	\hline 
	A^{(1)}_0&\Cdots&A_{p-1}^{(1)}&A_p^{(1)}&\Cdots&&A^{(1)}_{N-1}\\
	\Vdots&&\Vdots&\Vdots&&&\Vdots\\
	&&&&&&\\
	A^{(p)}_0&\Cdots&A_{p-1}^{(p)}&A_p^{(p)}&\Cdots&&A^{(p)}_{N-1}
\end{NiceArray}\right|=\left|\begin{NiceArray}{Wc{1.5cm}cc|cccc}
\Block{	4-3}<\Large>{0_{(N-p)\times p}}&&&\Block{4-4}<\Large>{I_{N-p}}&&&\\
&&&&&&\\
&&&&&&\\
&&&&&&\\
\hline 
\Block{3-3}<\Large>{\nu}&&&A_p^{(1)}&\Cdots&&A^{(1)}_{N-1}\\
&&&\Vdots&&&\Vdots\\
&&&&&&\\
&&&A_p^{(p)}&\Cdots&&A^{(p)}_{N-1}
\end{NiceArray}\right|=(-1)^{p(N-p)},
\end{align*}
where the initial conditions of recursion polynomials has been used, and we get
\begin{align*}
	(-1)^{p(N-p)}P_N(x)=(-T_{p,0})(-T_{p+1,1})\cdots(-T_{N-1,N-p-1})\det M
\end{align*}
so that
\begin{align*}
P_N(x)=(-1)^{(p+1)(N-p)}T_{p,0}T_{p+1,1}\cdots T_{N-1,N-p-1}	T_{N,N-p}\cdots T_{N+p-1,N-1}	\det A_N(x)
\end{align*}
and observing that $(-1)^{(p+1)(N-p)}=(-1)^{(p-1)N}$ we obtain the stated result.
Finally, the second result is proven analogously.
\end{proof}

\section{Right and left eigenvectors}
We now consider determinantal polynomials constructed in terms of determinants of left and right recursion polynomials that happen to give left and right eigenvectors of $T^{[N]}$.
\begin{defi}
	Let us introduce the determinantal polynomials
	\begin{align}\label{eq:QNn}
		Q_{n,N}&\coloneq\begin{vNiceMatrix}
			A^{(1)}_{n} & \Cdots & A^{(p)}_{n} \\[2pt]
			A^{(1)}_{N+1} & \Cdots & A^{(p)}_{N+1} \\[2pt]
			\Vdots & & \Vdots \\[2pt]
			A^{(1)}_{N+p-1} & \Cdots & A^{(p)}_{N+p-1}
		\end{vNiceMatrix},&	R_{n,N}&\coloneq\begin{vNiceMatrix}
		B^{(1)}_{n} & \Cdots & B^{(q)}_{n} \\[2pt]
		B^{(1)}_{N+1} & \Cdots & B^{(q)}_{N+1} \\[2pt]
		\Vdots & & \Vdots \\[2pt]
		B^{(1)}_{N+q-1} & \Cdots & B^{(q)}_{N+q-1}
	\end{vNiceMatrix},
	\end{align}
 the semi-infinite row and column vectors
\begin{align*}
		 Q_N&\coloneq\left[\begin{NiceMatrix}
		Q_{0,N} &Q_{1,N} &\Cdots
	\end{NiceMatrix}\right], & R_N&\coloneq\left[\begin{NiceMatrix}
	R_{0,N} \\R_{1,N} \\\Vdots
\end{NiceMatrix}\right],
\end{align*}
and corresponding truncations
	\begin{align*}
		Q^{\langle N\rangle}&\coloneq \begin{bNiceMatrix}
			Q_{0,N} &Q_{1,N}&\Cdots & Q_{N,N}
		\end{bNiceMatrix}, & 	R^{\langle N\rangle}&\coloneq \begin{bNiceMatrix}
		R_{0,N} \\R_{1,N}\\\Vdots \\ R_{N,N}
	\end{bNiceMatrix}.
	\end{align*}
\end{defi}

\begin{pro}\label{pro:properties_Q}
	The following properties for polynomials $Q_{n,N}, R_{n,N}$ are satisfied
	\begin{enumerate}
		\item $ Q_{N+1,N}=\cdots=Q_{N+p-1,N}=R_{N+1,N}=\cdots=R_{N+q-1,N}=0$.
		\item $\alpha_N Q_{N,N}=\beta_N R_{N,N}=P_N$ and $(-1)^{p-1}\alpha_{N+1}Q_{N+p,N}=(-1)^{q-1}\beta_{N+1}R_{N+q,N}=P_{N+1}$.
		\item $Q_NT=xQ_N$ and $TR_N=xR_N$.
		\item \begin{align}\label{eq:dual}
			Q^{\langle N\rangle} T^{[N]}+
			\begin{bNiceMatrix}
				0& \Cdots& 0& T_{N+p,N}Q_{N+p,N}
			\end{bNiceMatrix}&=x Q^{\langle N\rangle},&
		T^{[N]}R^{\langle N\rangle} +
		\begin{bNiceMatrix}
			0\\ \Vdots\\ 0\\ T_{N,N+q}R_{N+q,N}
		\end{bNiceMatrix}&=x R^{\langle N\rangle}.
		\end{align}
	\end{enumerate}
\end{pro}
\begin{proof}
	\begin{enumerate}
		\item As $Q_{n,N}$ and $R_{n,N}$ are the determinants in \eqref{eq:QNn} we see that it vanishes whenever two rows are equal, that happens precisely in the indicated cases.
		\item It follows from Theorem \ref{teo:determinants_mixed}. 
		\item It is a direct consequence of the fact that each appropriate rows/columns in the determinants in \eqref{eq:QNn} satisfies corresponding recurrences.
		\item It follows from the previous points i) and iii).
	\end{enumerate}
\end{proof}

Now, we are ready to give a set of left and right eigenvectors of the banded finite matrix $T^{[N]}$. Let us assume that its eigenvalues $\lambda_k^{[N]}$, $k\in\{1,\dots,N+1\}$ are simple (which happens for example for oscillatory matrices). These eigenvalues are the zeros of the characteristic polynomials $P_{N+1}(x)$.
\begin{pro}
	For $k\in\{1,\dots,N+1\}$, the vectors
$Q^{\langle N\rangle}\big|_{x=\lambda^{[N]}_k}$ and $R^{\langle N\rangle}\big|_{x=\lambda^{[N]}_k}$
	are left and right eigenvectors of $T^{[N]}$, respectively.
\end{pro}
\begin{proof}
	Properties ii) and iv) in Proposition \ref{pro:properties_Q} and an evaluation at $\lambda^{[N]}_k$ leads to the result.
\end{proof}
\section{Christoffel--Darboux formula}

We present now a generalized Christoffel--Darboux formula for the determinantal polynomials and the characteristic polynomial of a banded matrix. These results are an extension of the formulas found in \cite{Coussement-VanAssche} for the non-mixed case, see also \cite{bidiagonal_factorization_paper}. Christoffel--Darboux formulas, not of the type described here, for the mixed multiple orthogonality where discussed in \cite{Evi_Arno} and also in \cite{afm,Ariznabarreta_Manas}.
\begin{pro}[Christoffel--Darboux type formulas]\label{theorem:CD}
	\begin{enumerate}
		\item For the determinantal polynomials $Q_{n,N}$ and $R_{n,N}$ introduced in \eqref{eq:QNn} we get the following generalized Christoffel--Darboux formula
		\begin{align}\label{eq:CD2}
			\sum_{n=0}^{N}Q_{n,N}(x)R_{n,N}(y)= \frac{1}{\alpha_N \beta_N }\frac{P_{N+1}(x)P_{N}(y)-P_{N}(x)P_{N+1}(y)}{x-y}.
		\end{align}
		\item The following generalized confluent Christoffel--Darboux relation is fulfilled 
		\begin{align}\label{eq:CD2_confluent}
			\sum_{n=0}^{N}Q_{n,N}R_{n,N}= \frac{1}{\alpha_N \beta_N }\big(P'_{N+1}P_{N}-P'_{N}P_{N+1}\big).
		\end{align}
	\end{enumerate}
\end{pro}
\begin{proof}
				We use \eqref{eq:dual} to get 
				\begin{align*}
					& -Q^{\langle N\rangle}(x)
					\begin{bNiceMatrix}
						0\\ \Vdots\\ 0 \\ T_{N,N+q}R_{N+q,N} (y)
					\end{bNiceMatrix}+\begin{bNiceMatrix}
						0& \Cdots& 0& T_{N+p,N}Q_{N+p,N}(x)
					\end{bNiceMatrix} R^{\langle N\rangle }(y)=(x-y) Q^{\langle N\rangle}(x)R^{\langle N\rangle }(y),
				\end{align*}
				Now, recalling $Q_{N,N}= \alpha^{-1}_N P_N$,$Q_{N+p,N}= (-1)^{p-1}\alpha^{-1}_{N+1} P_{N+1}$, $\alpha_{N+1}= (-1)^{p-1} T_{N+p,N}\alpha_N $,
			$R_{N,N}= \beta^{-1}_N P_N$, $R_{N+q,N}= (-1)^{q-1}\beta^{-1}_{N+1} P_{N+1}$, $\beta_{N+1}= (-1)^{q-1} T_{N,N+q}\beta_N$,
				 we obtain \eqref{eq:CD2}. Finally, \eqref{eq:CD2_confluent} appears as a limit in \eqref{eq:CD2}.
			\end{proof}
			\enlargethispage{1cm}

			\section{Biorthogonality and Christoffel numbers}
			We now discuss, for the truncated situation, how to construct biorthogonal families of left and right eigenvectors and introduce the Christoffel numbers in this setting.
			
			\begin{defi}[Christoffel numbers]
				The Christoffel numbers or coefficients are defined as
				\begin{align*}
					\mu_{k,1}^{[N]}&\coloneq 
					\frac{\begin{vNiceMatrix}[small]
							A^{(2)}_{N+1} \big(\lambda^{[N]}_k\big)& \Cdots & A^{(p)}_{N+1}\big(\lambda^{[N]}_k\big) \\
							\Vdots & & \Vdots \\
							A^{(2)}_{N+p-1} \big(\lambda^{[N]}_k\big)& \Cdots & A^{(p)}_{N+p-1}\big(\lambda^{[N]}_k\big) 
					\end{vNiceMatrix}}{\beta_N\sum_{l=0}^{N}Q_{l,N}\big(\lambda^{[N]}_k\big)R_{l,N}\big(\lambda^{[N]}_k\big)},\\
					\mu^{[N]}_{k,2}&\coloneq - \frac{\begin{vNiceMatrix}[small]
							A^{(1)}_{N+1} \big(\lambda^{[N]}_k\big)& A^{(3)}_{N+1} \big(\lambda^{[N]}_k\big)&\Cdots & A^{(p)}_{N+1}\big(\lambda^{[N]}_k\big) \\[2pt]
							\Vdots & \Vdots & & \Vdots \\
							A^{(1)}_{N+p-1} \big(\lambda^{[N]}_k\big)& A^{(3)}_{N+p-1} \big(\lambda^{[N]}_k\big)&\Cdots & A^{(p)}_{N+p-1}\big(\lambda^{[N]}_k\big)
					\end{vNiceMatrix}}{\beta_N\sum_{l=0}^{N}Q_{l,N}\big(\lambda^{[N]}_k\big)R_{l,N}\big(\lambda^{[N]}_k\big)},\\
					&\hspace*{8pt}\vdots \\
					\mu^{[N]}_{k,p}&\coloneq (-1)^{p-1} \frac{\begin{vNiceMatrix}[small]
							A^{(1)}_{N+1} \big(\lambda^{[N]}_k\big)&\Cdots & A^{(p-1)}_{N+1}\big(\lambda^{[N]}_k\big) \\[2pt]
							\Vdots & & \Vdots \\
							A^{(1)}_{N+p-1} \big(\lambda^{[N]}_k\big)&\Cdots & A^{(p-1)}_{N+p-1}\big(\lambda^{[N]}_k\big)
					\end{vNiceMatrix}}{\beta_N\sum_{l=0}^{N}Q_{l,N}\big(\lambda^{[N]}_k\big)R_{l,N}\big(\lambda^{[N]}_k\big)},\\
					\rho_{k,1}^{[N]}&\coloneq \beta_N
					\begin{vNiceMatrix}[small]
						B^{(2)}_{N+1} \big(\lambda^{[N]}_k\big)& \Cdots & B^{(q)}_{N+1}\big(\lambda^{[N]}_k\big) \\
						\Vdots & & \Vdots \\
						B^{(2)}_{N+q-1} \big(\lambda^{[N]}_k\big)& \Cdots & B^{(p)}_{N+q-1}\big(\lambda^{[N]}_k\big) 
					\end{vNiceMatrix},\\
					\rho^{[N]}_{k,2}&\coloneq -\beta_N\begin{vNiceMatrix}[small]
						B^{(1)}_{N+1} \big(\lambda^{[N]}_k\big)& B^{(3)}_{N+1} \big(\lambda^{[N]}_k\big)&\Cdots & B^{(q)}_{N+1}\big(\lambda^{[N]}_k\big) \\[2pt]
						\Vdots & \Vdots & & \Vdots \\
						B^{(1)}_{N+q-1} \big(\lambda^{[N]}_k\big)& B^{(3)}_{n+q-1} \big(\lambda^{[N]}_k\big)&\Cdots & B^{(p)}_{N+q-1}\big(\lambda^{[N]}_k\big) 
					\end{vNiceMatrix},\\
					&\hspace*{8pt}\vdots \\
					\rho^{[N]}_{k,q}&\coloneq (-1)^{q-1} \beta_N\begin{vNiceMatrix}[small]
						B^{(1)}_{N+1} \big(\lambda^{[N]}_k\big)&\Cdots & B^{(q-1)}_{N+1}\big(\lambda^{[N]}_k\big) \\
						\Vdots & & \Vdots \\[2pt]
						B^{(1)}_{N+q-1} \big(\lambda^{[N]}_k\big)&\Cdots & B^{(q-1)}_{N+q-1}\big(\lambda^{[N]}_k\big)
					\end{vNiceMatrix}.
				\end{align*}
			\end{defi}
			
			\begin{pro}[Spectral properties]\label{pro:UW} Assume that $P_{N+1}$ has simple zeros at the set $\big\{\lambda^{[N]}_k\big\}_{k=1}^{N+1}$, so that
				the vectors $u_k^{\langle N\rangle}\coloneq R^{\langle N\rangle}\big(\lambda^{[N]}_k\big)$ and $\tilde w^{\langle N\rangle}_k\coloneq Q^{\langle N\rangle} \big(\lambda^{[N]}_k\big)$ are right and left eigenvectors of $ T^{[N]}$, respectively, $k=1,\dots,N+1$. Then:
				\begin{enumerate}
					\item Biorthogonal families left and right eigenvectors $\big\{w^{\langle N\rangle}_k\big\}_{k=1}^{N+1}$ and $\big\{u^{\langle N\rangle}_k\big\}_{k=1}^{N+1}$, are
					\begin{align*}
						w^{\langle N\rangle}_{k}&=\frac{ Q^{\langle N\rangle}\big(\lambda^{[N]}_k\big)}{\beta_N\sum_{l=0}^{N}Q_{l,N}\big(\lambda^{[N]}_k\big)R_{l,N}\big(\lambda^{[N]}_k\big)}, &
							u^{\langle N\rangle}_{k}&=\beta_N R^{\langle N\rangle}\big(\lambda^{[N]}_k\big).
					\end{align*}
					\item The following expression holds
					\begin{align}
						\label{eq:kcomponentelefteigenI}
						w^{\langle N\rangle}_{k,n}&=
						\frac{ \alpha_N Q_{n-1,N}\big(\lambda^{[N]}_k\big)
						}{
							P_{N}\big(\lambda^{[N]}_k\big)P'_{N+1}\big(\lambda^{[N]}_k\big)}, & 	u^{\langle N\rangle}_{k,n}&=\beta_N R_{n-1,N}\big(\lambda^{[N]}_k\big).
					\end{align}
					\item In terms of the Christoffel numbers we can write
					\begin{align}\label{eq:discrete_linear_form_typeI}
						w^{\langle N\rangle}_{k,n}= A_{n-1}^{(1)}\big(\lambda^{[N]}_k\big)\mu^{[N]}_{k,1} +\cdots +A_{n-1}^{(p)}\big(\lambda^{[N]}_k\big)\mu^{[N]}_{k,p},\\
				\label{eq:discrete_linear_form_typeII}
				u^{\langle N\rangle}_{k,n}= B_{n-1}^{(1)}\big(\lambda^{[N]}_k\big)\rho^{[N]}_{k,1} +\cdots +B_{n-1}^{(q)}\big(\lambda^{[N]}_k\big)\rho^{[N]}_{k,q}.
				\end{align}
					\item For the Christoffel numbers it holds that
					\begin{align}
						\label{eq:kcomponentelefteigenII}
						\begin{bNiceMatrix}
							\mu^{[N]}_{k,1} \\[5pt]
							\mu^{[N]}_{k,2} \\
							\Vdots
							\\
							\mu^{[N]}_{k,p}
						\end{bNiceMatrix}
						&= \begin{bNiceMatrix}
							1& 0 & \Cdots& && 0 \\
							\nu^{(1)}_1 & 1 & \Ddots&& & \Vdots \\
							\Vdots & \Ddots & \Ddots& && \\
							&&&& &\\&&&&&0\\
							\nu^{(1)}_{p-1} &\Cdots& && \nu^{(p-1)}_{p-1}& 1 
						\end{bNiceMatrix}^{-1} \begin{bNiceMatrix}
							w^{\langle N\rangle}_{k,1} \\[5pt]
							w^{\langle N\rangle}_{k,2} \\
							\Vdots \\
							w^{\langle N\rangle}_{k,p}
						\end{bNiceMatrix}, &
							\begin{bNiceMatrix}
						\rho^{[N]}_{k,1} \\[5pt]
						\rho^{[N]}_{k,2} \\
						\Vdots
						\\
						\rho^{[N]}_{k,q}
					\end{bNiceMatrix}
					&= \begin{bNiceMatrix}
						1& 0 & \Cdots& && 0 \\
						\xi^{(1)}_1 & 1 & \Ddots&& & \Vdots \\
						\Vdots & \Ddots & \Ddots& && \\
						&&&& &\\&&&&&0\\
						\xi^{(1)}_{q-1} &\Cdots& && \xi^{(q-1)}_{q-1}& 1 
					\end{bNiceMatrix}^{-1} \begin{bNiceMatrix}
						u^{\langle N\rangle}_{k,1} \\[5pt]
						u^{\langle N\rangle}_{k,2} \\
						\Vdots \\
						u^{\langle N\rangle}_{k,q}
					\end{bNiceMatrix}.
					\end{align} 
					\item The corresponding matrices $\mathscr U$ (with columns the right eigenvectors $u_k$ arranged in the standard order) and 
					$\mathscr W$ (with rows the left eigenvectors $w_k$ arranged in the standard order) satisfy
					\begin{align}\label{eq:UW=I}
						\mathscr U\mathscr W=\mathscr W\mathscr U=I_{N+1}.
					\end{align}
					\item In terms of the eigenvalues diagonal matrix $D=\diag\big(\lambda^{[N]}_1,\dots,\lambda^{[N]}_{N+1}\big)$ we have
					\begin{align}\label{eq:UDnW=Jn}
						\mathscr UD^n\mathscr W&=\big(T^{[N]}\big)^n, & n&\in\N_0.
					\end{align}
				\end{enumerate}
			\end{pro}
			\begin{proof}
				\begin{enumerate}
					\item As the zeros are simple we have that left and right eigenvectors are orthogonal, i.e., $\tilde w^{\langle N\rangle}_k u^{\langle N\rangle}_l=\delta_{k,l} \sum_{r=0}^{N}Q_{r,N}\big(\lambda^{[N]}_k\big)R_{r,N}\big(\lambda^{[N]}_k\big)$. Hence, we divide by $\sum_{r=0}^{N}Q_{r,N}\big(\lambda^{[N]}_k\big)R_{r,N}\big(\lambda^{[N]}_k\big)$ to get normalized left eigenvectors.
					\item It follows from the previous result and Equation \eqref{eq:CD2_confluent}. 
					\item 
					In Equation \eqref{eq:QNn} expand the determinant in $Q_{n-1,N}$ along its first row.
					\item Use \eqref{eq:discrete_linear_form_typeI} for the first $p$ entries
					\begin{align*}
						\begin{bNiceMatrix}
							w^{\langle N\rangle}_{k,1} \\
							\Vdots \\
							w^{\langle N\rangle}_{k,p}
						\end{bNiceMatrix} 
						&= \begin{bNiceMatrix}
							A^{(1)}_{0} \big(\lambda^{[N]}_k\big)&\Cdots & A^{(p)}_{0}\big(\lambda^{[N]}_k\big) \\[2pt]
							\Vdots & & \Vdots \\[2pt]
							A^{(1)}_{p-1} \big(\lambda^{[N]}_k\big)&\Cdots & A^{(p)}_{p-1}\big(\lambda^{[N]}_k\big) \\[2pt]
						\end{bNiceMatrix} \begin{bNiceMatrix}
							\mu^{[N]}_{k,1} \\
							\Vdots
							\\
							\mu^{[N]}_{k,p}
						\end{bNiceMatrix}
					\end{align*} 
					and the initial conditions~\eqref{eq:initcondtypeI}
					\begin{align*}
						\begin{bNiceMatrix}
							A^{(1)}_{0} \big(\lambda^{[N]}_k\big)&\Cdots & A^{(p)}_{0}\big(\lambda^{[N]}_k\big) \\[2pt]
							A^{(1)}_{1} \big(\lambda^{[N]}_k\big)&\Cdots & A^{(p)}_{1}\big(\lambda^{[N]}_k\big) \\[2pt]
							\Vdots & & \Vdots \\[2pt]
							A^{(1)}_{p-1} \big(\lambda^{[N]}_k\big)&\Cdots & A^{(p)}_{p-1}\big(\lambda^{[N]}_k\big) \\[2pt]
						\end{bNiceMatrix} = \begin{bNiceMatrix}
							1& 0 & \Cdots& && 0 \\
							\nu^{(1)}_1 & 1 & \Ddots&& & \Vdots \\
							\Vdots & \Ddots & \Ddots& && \\
							&&&& &\\&&&&&0\\
							\nu^{(1)}_{p-1} &\Cdots& && \nu^{(p-1)}_{p-1}& 1 
						\end{bNiceMatrix}
					\end{align*}
					to obtain the result. For the right vectors proceed similarly.
					\item It follows from the biorthogonality of the left and right eigenvectors.
					\item Notice that $\mathscr UD^n =\big(T^{[N]}\big)^n\mathscr U $ and use $\mathscr U^{-1}=\mathscr W$ to get $\mathscr UD^n\mathscr W=\big(T^{[N]}\big)^n$ as desired.
				\end{enumerate}
			\end{proof}
		
		\section{Mixed multiple discrete orthogonality}
		
		We reformulate the previous discussed biorthogonality in terms of a set of discrete measures and corresponding mixed multiple discrete orthogonality.			
			\begin{defi}[Step functions]\label{def:Stieltjes}
	Let us consider the following step  functions
	\begin{align*}
		\psi^{[N]}_{b,a}&\coloneq \begin{cases}
			0, & x<\lambda^{[N]}_{N+1},\\[2pt]
			\rho^{[N]}_{1,b}\mu^{[N]}_{1,a}+\cdots+\rho^{[N]}_{k,b}\mu^{[N]}_{k,a}, & \lambda^{[N]}_{k+1}\leqslant x< \lambda^{[N]}_{k}, \quad k\in\{1,\dots,N\},\\[2pt]
			\rho^{[N]}_{1,b}\mu^{[N]}_{1,a}+\cdots+\rho^{[N]}_{N+1,b}\mu^{[N]}_{N+1,a} ,
			& x \geqslant \lambda^{[N]}_{1}.
		\end{cases}
						\end{align*}
			\end{defi}
		
				We now show that last step of these step functions is  bounded. This implies in the case of positive Christoffel coefficients  that these step  functions are uniformly bounded in $N$. For that aim we need to introduce the matrix $I_{q,p}\in\R^{q\times p}$, with $(I_{q,p})_{k,l}=\delta_{k,l}$. Thus, if $p=q$ we are dealing with the identity matrix, however if $p\neq q$ is a rectangular matrix with a square block with the identity $I_{\min(p,q)}$ completed with a zero block.
			\begin{pro}\label{pro:boudness_Christoffel}
			For $a\in\{1, \dots, p\}$ and $b\in\{1, \dots, q\}$, we have 
			\begin{align}\label{eq:bound}
				\rho^{[N]}_{1,b}\mu^{[N]}_{1,a}+\cdots+\rho^{[N]}_{N+1,b}\mu^{[N]}_{N+1,a} = (\xi^{-1}I_{q,p}\nu^{-\top})_{b,a}.
			\end{align}
		\end{pro}
		\begin{proof}
			Let us write \eqref{eq:kcomponentelefteigenII} in the alternative form
			\begin{align}\label{eq:mu_W_e_nu}
				\begin{bNiceMatrix}
					\mu^{[N]}_{1,1} &\Cdots&\mu^{[N]}_{1,p} \\
					\Vdots & & \Vdots \\
					\mu^{[N]}_{N+1,1} &\Cdots&\mu^{[N]}_{N+1,p}
				\end{bNiceMatrix}
				&= \begin{bNiceMatrix}
					w^{\langle N\rangle}_{1,1} &\Cdots &w^{\langle N\rangle}_{1,p} \\
					\Vdots & & \Vdots \\
					w^{\langle N\rangle}_{N+1,1} &\Cdots &w^{\langle N\rangle}_{N+1,p}
				\end{bNiceMatrix} \begin{bNiceMatrix}
					1& 0 & \Cdots& && 0 \\
					\nu^{(1)}_1 & 1 & \Ddots&& & \Vdots \\
					\Vdots & \Ddots & \Ddots& && \\
					&&&& &\\&&&&&0\\
					\nu^{(1)}_{p-1} &\Cdots& && \nu^{(p-1)}_{p-1}& 1 
				\end{bNiceMatrix}^{-\top} ,\\
				\label{eq:rho_U_e_nu}
				\begin{bNiceMatrix}
					\rho^{[N]}_{1,1} &\Cdots&\rho^{[N]}_{N+1,1} \\
					\Vdots & & \Vdots \\
					\rho^{[N]}_{1,q} &\Cdots&\rho^{[N]}_{N+1,q}
				\end{bNiceMatrix}
				&= \begin{bNiceMatrix}
					1& 0 & \Cdots& && 0 \\
					\xi^{(1)}_1 & 1 & \Ddots&& & \Vdots \\
					\Vdots & \Ddots & \Ddots& && \\
					&&&& &\\&&&&&0\\
					\xi^{(1)}_{q-1} &\Cdots& && \xi^{(q-1)}_{q-1}& 1 
				\end{bNiceMatrix}^{-1}\begin{bNiceMatrix}
					u^{\langle N\rangle}_{1,1} &\Cdots &u^{\langle N\rangle}_{N+1,1} \\
					\Vdots & & \Vdots \\
					u^{\langle N\rangle}_{1,q} &\Cdots &u^{\langle N\rangle}_{N+1,q}
				\end{bNiceMatrix}.
			\end{align} 
			From $\mathscr U\mathscr W = I$, we obtain
			\begin{align*}
				\begin{bNiceMatrix}
					u^{\langle N\rangle}_{1,1} &\Cdots &u^{\langle N\rangle}_{N+1,1} \\
					\Vdots & & \Vdots \\
					u^{\langle N\rangle}_{1,q} &\Cdots &u^{\langle N\rangle}_{N+1,q}
				\end{bNiceMatrix}	\begin{bNiceMatrix}
					w^{\langle N\rangle}_{1,1} &\Cdots &w^{\langle N\rangle}_{1,p} \\
					\Vdots & & \Vdots \\
					w^{\langle N\rangle}_{N+1,1} &\Cdots &w^{\langle N\rangle}_{N+1,p}
				\end{bNiceMatrix} =I_{q,p}.
			\end{align*}
			Hence,
			\begin{align}\label{eq:rho_mu_xi_nu}
				\begin{bNiceMatrix}
					\rho^{[N]}_{1,1} &\Cdots&\rho^{[N]}_{N+1,1} \\
					\Vdots & & \Vdots \\
					\rho^{[N]}_{1,q} &\Cdots&\rho^{[N]}_{N+1,q}
				\end{bNiceMatrix}			\begin{bNiceMatrix}
					\mu^{[N]}_{1,1} &\Cdots&\mu^{[N]}_{1,p} \\
					\Vdots & & \Vdots \\
					\mu^{[N]}_{N+1,1} &\Cdots&\mu^{[N]}_{N+1,p}
				\end{bNiceMatrix}=\xi^{-1} I_{q,p}\nu^{-\top}
			\end{align}
			and we get $\mu^{[N]}_{1,a}\rho^{[N]}_{1,b}+\cdots+\mu^{[N]}_{N+1,a}\rho^{[N]}_{N+1,b} = (\xi^{-1}I_{q,p}\nu^{-\top})_{b,a}$.
		\end{proof}
	Notice that these functions have bounded variation and are right continuous, so it make sense to consider the associated Lebesgue--Stieltjes measures.
						\begin{defi}[Matrix of discrete measures]
				Let us introduce a 
		$q\times p$ matrix  				$\Psi^{[N]}\coloneq\begin{bNiceMatrix}[small]
									\psi^{[N]}_{1,1}&\Cdots &\psi^{[N]}_{1,p}\\
									\Vdots & &\Vdots\\
									\psi^{[N]}_{q,1}&\Cdots &\psi^{[N]}_{q,p}
								\end{bNiceMatrix}$
and the corresponding 
				$q\times p$ matrix of 
				discrete Lebesgue--Stieltjes measures supported at the zeros of $P_{N+1}$, 
				\begin{align}\label{eq:discrete_mesures}
					\d\Psi^{[N]}=\begin{bNiceMatrix}
						\d\psi ^{[N]}_{1,1}&\Cdots &	\d\psi^{[N]}_{1,p}\\
						\Vdots & &\Vdots\\
							\d\psi^{[N]}_{q,1}&\Cdots &	\d\psi^{[N]}_{q,p}
					\end{bNiceMatrix}=\sum_{k=1}^{N+1}\begin{bNiceMatrix}
						\rho^{[N]}_{k,1}\\\Vdots \\	\rho^{[N]}_{k,q}
					\end{bNiceMatrix}\begin{bNiceMatrix}
						\mu^{[N]}_{k,1} & \Cdots & \mu^{[N]}_{k,p}
					\end{bNiceMatrix}\delta\big(x-\lambda^{[N]}_k\big).
				\end{align}
			\end{defi}
			\begin{rem}
				This matrix of discrete measures is rank $1$ at each point of the support.
			\end{rem}
	
				\begin{teo}[Mixed multiple discrete biorthogonality]\label{theorem:birothoganality}
				Assume that the recursion polynomials $P_{N+1}$ have simple zeros $\big\{\lambda_k^{[N]}\big\}_{k=1}^{N+1}$.
				The following biorthogonal relations hold
				\begin{align*}
					\sum_{a=1}^p\sum_{b=1}^{q}	\int B^{(b)}_n(x)\d\psi^{[N]}_{b,a}(x)A^{(a)}_{m}(x)&=\delta_{n,m}, &n,m&\in\{0,\dots,N\}.
				\end{align*}
			\end{teo}
			\begin{proof}
				It follows from Equations \eqref{eq:discrete_linear_form_typeI}, \eqref{eq:discrete_linear_form_typeII} and 
				$\mathscr U\mathscr W=I$.
			\end{proof}			
			From this biorthogonality we get the following:
						\begin{coro}[Mixed multiple discrete orthogonality]
				Assume that the polynomial $P_{N+1}$ has simple zeros $\big\{\lambda_k^{[N]}\big\}_{k=1}^{N+1}$.
				Then, the following discrete type mixed multiple orthogonality for $m \in\{1,\dots, N\}$ are satisfied:
				\begin{align*}
					\sum_{a=1}^p\int x^n\d\psi ^{[N]}_{b,a}A^{(a)}_{m}&=0, & 
					n&\in\big\{0,\dots,\deg B^{(b)}_{m-1}\big\},&	b&\in\{1,\dots,q\},\\
					\sum_{b=1}^q\int B^{(b)}_{m}	\d\psi ^{[N]}_{b,a}x^n&=0, 
					&				n&\in\big\{0,\dots,\deg A^{(a)}_{m-1}\big\}, & a&\in\{1,\dots,p\}.
				\end{align*}
			\end{coro}
			
			\section{Positive bidiagonal factorization and Christoffel numbers positivity}
			Positive bidiagonal factorization (PBF) accommodate naturally to TN banded matrices as all the subdiagonals may be constructed in terms of simpler bidiagonal matrices. 
			\begin{defi}[Positive bidiagonal factorization]
				We say that a banded matrix $ T$ as in \eqref{eq:monic_Hessenberg} admits a PBF if
				\begin{align}\label{eq:bidiagonal}
					T= L_{1} \cdots L_{p} \Delta U_q\cdots U_1,
				\end{align}
				with $\Delta=\diag(\Delta_0,\Delta_1,\dots)$ and bidiagonal matrices given respectively by 
		
				\begin{align}\label{eq:bidiagonal_factors}
						L_k&\coloneq \left[\begin{NiceMatrix}[columns-width=auto]
							1 &0&\Cdots&\\
							 L_{k|0} & 1 &\Ddots&\\
							0& L_{k|1}& 1& \\
							\Vdots&\Ddots& \Ddots& \Ddots\\&&&
						\end{NiceMatrix}\right], & 
						U_k& \coloneq
						\left[\begin{NiceMatrix}[columns-width = auto]
							1& U_{k|0}&0&\Cdots&\\
							0& 1& U_{k|1}&\Ddots&\\
							\Vdots&\Ddots&1&\Ddots&\\
							& &\Ddots &\Ddots &\\&&&&
						\end{NiceMatrix}\right], 
				\end{align}
				and such that the positivity constraints $L_{k|i}, U_{k|i}, \Delta_i>0,$ for $i\in\N_0$, are satisfied. 
			\end{defi}

			\begin{rem}
				Notice that $L_1,\dots, L_p,\Delta^{[N]},U_q^{[N]},\dots ,U_1^{[N]}\in \operatorname{InTN}$. 
			\end{rem}
			\begin{pro}
				The above positive bidiagonal factorization of $T$ induces the following positive bidiagonal factorization for the leading principal submatrix $ T^{[N]}$ 
				\begin{align}\label{eq:bidiagonal_HessenbergN}
					T^{[N]}= L^{[N]}_{1} \cdots L_{p}^{[N]} \Delta^{[N]}U_q^{[N]}\cdots U_1^{[N]}.
				\end{align}
			\end{pro}

			\begin{pro}
				If $T$ has a PBF then its leading principal submatrices 	$T^{[N]} $ are oscillatory.
			\end{pro}
			\begin{proof}
				As all factors are InTN the product matrix is InTN. Moreover, as all parameters in the bidiagonal factors are positive then using Gantmacher--Krein Criterion we get that the matrix is oscillatory.
			\end{proof}

			
				\begin{pro}[Interlacing]\label{pro:interlacing}
				Let us assume that $T$ is oscillatory. Then:
					\begin{enumerate}
						\item The polynomial $P_{N+1}$ interlaces $P_N$.
						\item For $x\in\R$, for the corresponding Wronskian we find $P_{N+1}'P_N-P_N'P_{N+1}>0$. 
					In particular,
					\begin{align*}
						(P_{N+1}'P_N)\big|_{x=\lambda^{[N]}_k}&>0, & (P_{N+1}P_N')\big|_{x=\lambda^{[N-1]}_k}&<0.
					\end{align*}
						\item The confluent kernel is a positive function; i.e., $\alpha_N \beta_N \sum_{n=0}^{N}Q_{n,N}(x)R_{n,N}(x)>0$ for $x\in\R$.
					\end{enumerate}
				\end{pro}
	\begin{proof}
	\begin{enumerate}
		\item 				Given that $T^{[N]}$ is oscillatory the polynomial $P_{N+1}$ interlaces $P_N$, see Theorem \ref{teo:eigennvalues_TN_2}.
		\item As the polynomials interlace its Wronskian $P_{N+1}'P_N-P_N'P_{N+1}$ has constant sign for $x\in\R$\footnote{In terms of $\pi^{[N]}_k\coloneq\frac{P_{N+1}}{x-\lambda^{[N]}_k}=\prod_{l\neq k}\big(x-\lambda^{[N]}_l\big)$ we have $P_N=\sum_{k=1}^{N+1} b_k \pi^{[N]}_k$ with $b_l=\frac{P_N(\lambda^{[N]}_l)}{\pi_l^{[N]}(\lambda^{[N]}_l)}\neq 0$ and that, as these polynomials interlaces, all the $b_k$ have the same sign; indeed, $\sgn P_N(\lambda^{[N]}_l)=-\sgn P_N(\lambda^{[N]}_{l+1})$ by interlacing and $\sgn \pi_l^{[N]}(\lambda^{[N]}_l)=-\sgn \pi^{[N]}_{l+1}(\lambda^{[N]}_{l+1})$ by definition. Consequently,  $\frac{P_N}{P_{N+1}}=\sum_{k=1}^{N+1}\frac{b_k}{x-\lambda_k^{[N]}}$, so that
		$P_{N+1}'P_N-P_N'P_{N+1}=P_{N+1}^2\Big(\frac{P_N}{P_{N+1}}\Big)'=-P_{N+1}^2\sum_{k=1}^{N+1}\frac{b_k}{(x-\lambda_k^{[N]})^2}=
		-\sum_{k=1}^{N+1}b_k (\pi^{[N]}_k)^2$, and the result follows.} 
	and, as the characteristic polynomials are monic, we have that $P_{N+1}'P_N-P_N'P_{N+1}=x^{2N}+O(x^{2N-1})$ for $|x|\to\infty$. Hence, the Wronskian is positive and 
		$P_N\big(\lambda^{[N]}_k\big)P_N'\big(\lambda^{[N]}_k\big)>0$ and $P_N'\big(\lambda^{[N-1]}_k\big)P_N\big(\lambda^{[N-1]}_k\big)>0$. 
		\item From \eqref{eq:CD2_confluent} we get $\sum_{n=0}^{N}Q_{n,N}R_{n,N}= \frac{1}{\alpha_N \beta_N }\big(P'_{N+1}P_{N}-P'_{N}P_{N+1}\big)$ and the result follows immediately.
	\end{enumerate}
\end{proof}

				We now explore some consequences that a positive bidiagonal factorization has. For that aim we introduce the idea of Darboux transformation of a banded Hessenberg matrix. Darboux transformations for banded Hessenberg matrices (beyond the tridiagonal situation) were discussed in \cite{dolores}. In \cite{proximo} for the tetradiagonal case, and corresponding multiple orthogonal polynomials in the step-line with two weights, the PBF factorization is given in terms of the values of the orthogonal polynomials of type I and II at $0$ and, consequently, an spectral interpretation of the Darboux transformation is given.
			\begin{defi}[Darboux transformations of banded matrices]
				Let us assume that $T$ admits a bidiagonal factorization (not necessarily positive). For each of its truncations $T^{[N]}$ we consider a chain of new auxiliary matrices, called Darboux transformations, given by the consecutive permutation of the unitriangular matrices in the factorization~\eqref{eq:bidiagonal_HessenbergN},
				\begin{align}\label{eq:Darboux+}
					\left\{\begin{aligned}
						\hat T^{[N,+1]}&=L_2 ^{[N]}\cdots L_p ^{[N]} \Delta^{[N]} U_q^{[N]} \cdots U_1^{[N]} L_1^{[N]}, \\
						\hat T^{[N,+2]}&=L_3 ^{[N]}\cdots L_p ^{[N]} \Delta^{[N]} U_q^{[N]} \cdots U_1^{[N]}L_1^{[N]}L_2 ^{[N]}, \\ &\hspace*{5pt} \vdots\\
						\hat T^{[N,+(p-1)]}&=L_p ^{[N]} \Delta^{[N]} U_q^{[N]} \cdots U_1^{[N]}L_1^{[N]}L_2 ^{[N]} \cdots L_{p-1} ^{[N]},\\
						\hat T^{[N,+p]}&= \Delta^{[N]} U_q^{[N]} \cdots U_1^{[N]}L_1^{[N]}L_2 ^{[N]} \cdots L_{p} ^{[N]},
					\end{aligned}\right.
				\end{align}
			and 
					\begin{align}\label{eq:Darboux+}
				\left\{\begin{aligned}
					\hat T^{[N,-1]}&=U_1^{[N]}L_1 ^{[N]}\cdots L_p ^{[N]} \Delta^{[N]} U_q^{[N]} \cdots U_2^{[N]} , \\
					\hat T^{[N,-2]}&=U_2^{[N]}U_1^{[N]}L_1^{[N]}\cdots L_p ^{[N]} \Delta^{[N]} U_q^{[N]} \cdots U_3^{[N]}, \\ &\hspace*{5pt} \vdots\\
					\hat T^{[N,-(q-1)]}&= U_{q-1}^{[N]} \cdots U_1^{[N]}L_1^{[N]}L_2 ^{[N]} \cdots L_{p} ^{[N]}\Delta^{[N]}U_{q}^{[N]} ,\\
	\hat T^{[N,-q]}&= U_{q}^{[N]} \cdots U_1^{[N]}L_1^{[N]}L_2 ^{[N]} \cdots L_{p} ^{[N]}\Delta^{[N]} .
			\end{aligned}\right.
			\end{align}
			\end{defi}
			\begin{lemma}
				 Darboux transformations are banded matrices with only its first $p$ subdiagonals, main diagonal and $q$ superdiagonals possibly different from zero.
		If $T$ admits a PBF then entries in these diagonals are positive.	\end{lemma}
		\begin{proof}
			It is a simple computation recalling the positivity of the nonzero entries.
		\end{proof}
		\begin{lemma}\label{lemma:poly}
			Let us assume that $T$ has a PBF.
			Then, for $k\in\{1,\dots,p\}$, we find:
			\begin{enumerate}
				\item The Darboux transformations $\hat{ T}^{[N,+a]}$, $a\in\{1,\dots, p\}$, $\hat{ T}^{[N,-b]}$, $b\in\{1,\dots, q\}$ are oscillatory.
				\item The characteristic polynomial of the Darboux transformations 
				$\hat{ T}^{[N,+a]}$, $a\in\{1,\dots, p\}$, $\hat{ T}^{[N,-b]}$, $b\in\{1,\dots, q\}$ is $P_{N+1}$. 
				\item If $w,u$ are left and right eigenvectors of $ T^{[N]}$, respectively, then $\hat w=w L_{1}^{[N]}\cdots L_{a}^{[N]}$ is a left eigenvector of $\hat{ T}^{[N,+a]}$
				 and $\hat u=U_{b}^{[N]}\cdots U_{1}^{[N]}u$ is a right eigenvector of $\hat{ T}^{[N,-b]}$.
			\end{enumerate}
		\end{lemma}
		\begin{proof}
			\begin{enumerate}
				\item Each bidiagonal factor belongs to InTN. Then, the Darboux transformation $\hat{ T}^{[N,k]}$ is a product of matrices in InTN and, consequently, belongs to InTN. Moreover, the entries in the first superdiagonal are all $1$, while the entries in the first subdiagonal are sum of products of $\alpha$'s. Thus, all entries in these two diagonals are positive. According to Gantmacher--Krein Criterion is an oscillatory matrix.
				\item As
				${\hat T}^{[N,+a]}=( L_{1}^{[N]}\cdots L_{a}^{[N]} )^{-1} T^{[N]} L_{1}^{[N]}\cdots L_{a}^{[N]} $ its characteristic polynomial is $P_{N+1}$. 
				Similarly, as 	${\hat T}^{[N,-b]}=U_{b}^{[N]}\cdots U_{1}^{[N]}T^{[N]} (U_{b}^{[N]}\cdots U_{1}^{[N]})^{-1} $ the corresponding characteristic polynomial is again $P_{N+1}$.
				
				\item
				We see that
				\begin{align*}
					\lambda \hat w&=\lambda w L_{1}^{[N]}\cdots L_{a}^{[N]} =w L^{[N]}_{1} \cdots L_{a}^{[N]} L^{[N]}_{a+1} \cdots L_{p}^{[N]} \Delta^{[N]}U_q^{[N]}\cdots U_1^{[N]}L_{1}^{[N]}\cdots L_{a}^{[N]} =\hat w \hat{ T}^{[N]},\\
					\lambda \hat u&=\lambda U_{b}^{[N]}\cdots U_{1}^{[N]} u= U^{[N]}_{b} \cdots U^{[N]}_{1} L^{[N]}_{1} \cdots L_{p}^{[N]}\Delta^{[N]}U_q^{[N]}\cdots U_{b+1}^{[N]} U_b^{[N]}\cdots U_1^{[N]} u= \hat{ T}^{[N]}\hat u.
				\end{align*}
			\end{enumerate}
		\end{proof}
		
		In order to show the positivity of the Christoffel coefficients we require of some preliminary notation.
		
		\begin{defi}\label{def:L_alpha}
			Let us define the matrices
			\begin{align*}
				 \Lambda&\coloneq\begin{bNiceMatrix}
					\Lambda^{(1)}& \Cdots& \Lambda^{(p)}
				\end{bNiceMatrix}\in\R^{p\times p}, &
			\Upsilon&\coloneq \begin{bNiceMatrix}
				\Upsilon^{(1)}\\\Vdots\\ \Upsilon^{(q)}
			\end{bNiceMatrix}\in\R^{q\times q}, 
			\end{align*}
			with 
			\begin{align*}
				\Lambda^{(1)} &\coloneq \begin{bNiceMatrix}
					1\\0\\\Vdots\\0
				\end{bNiceMatrix},& \Lambda^{(k)}&\coloneq\frac{1}{r_k}
				L_1^{[p-1]} \cdots L_{k-1}^{[p-1]} 
				\begin{bNiceMatrix}
					1\\0\\\Vdots\\0
				\end{bNiceMatrix}, & r_k\coloneq 
				& L_{k|0}L_{k-1|1}\cdots L_{1|k-1},
				& k&\in\{2,\dots, p\},
			\end{align*}
		and
		\begin{align*}
\hspace*{-.5cm}\begin{aligned}
					\Upsilon^{(1)}&\coloneq \begin{bNiceMatrix}
					1&0&\Cdots &0
				\end{bNiceMatrix},& \Upsilon^{(k)}&\coloneq\frac{1}{s_k}
				\begin{bNiceMatrix}
				1& 0 &\Cdots&0
			\end{bNiceMatrix}	U_1^{[q-1]} \cdots U_{k-1}^{[q-1]} 
			, & s_k\coloneq 
				& U_{k|0}U_{k-1|1}\cdots U_{1|k-1},
				& k&\in\{2,\dots, q\}.
\end{aligned}
			\end{align*}
		\end{defi}
		
		\begin{lemma}
		The matrices $\Lambda$ and $\Upsilon$ are positive upper and lower unitriangular matrices, respectively.
		\end{lemma}
	
	\begin{teo}[Christoffel numbers positivity]\label{teo:positivity}
		Let us assume that $T$ has a PBF and choose the matrices of initial conditions as
		\begin{align}\label{eq:tunnig_initial_conditions}
			\nu^{-\top}&=	\Lambda\mathcal A, &\xi^{-1}&=\mathcal B	\Upsilon, 
		\end{align}
		for some upper and lower unitriangular nonnegative matrices $\mathcal A\in\R^{p\times p}$ and $\mathcal B\in\R^{q\times q}$, respectively. Then, 
		\begin{align*}
			\rho^{[N]}_{k,b}&>0, &\mu^{[N]}_{k,a}&>0, &k&\in{1,\dots,N+1}, &a&\in\{1,\dots, p\}, & b&\in\{1,\dots,q\}.
		\end{align*}
	\end{teo}
	\begin{proof}
				Recall that the Christoffel numbers can expressed, in terms of the initial condition matrices $\xi$ and $\nu$, through the formulas
		\begin{align*}
		\begin{bNiceMatrix}
		\mu^{[N]}_{k,1} &\Cdots & 	\mu^{[N]}_{k,p}
	\end{bNiceMatrix}		&= 	\begin{bNiceMatrix}
			w^{\langle N\rangle}_{k,1} &\Cdots & 	w^{\langle N\rangle}_{k,p}
		\end{bNiceMatrix}\nu^{-\top}, &
	\begin{bNiceMatrix}
		\rho^{[N]}_{k,1} \\\Vdots \\	\rho^{[N]}_{k,q}
	\end{bNiceMatrix}&	= 	\xi^{-1}\begin{bNiceMatrix}
	u^{\langle N\rangle}_{k,1} \\\Vdots \\	u^{\langle N\rangle}_{k,q}
\end{bNiceMatrix},
		\end{align*}
	that relates these Christoffel numbers with the corresponding biorthogonal families of right an left eigenvectors.
	Notice that the entries of these biorthogonal right and left eigenvectors can be written as
$w^{\langle N\rangle}_{k,a}=\alpha_N \left.\frac{Q_{a-1,N}}{P_{N+1}'P_N}\right|_{x=\lambda_k^{[N]}}$ and $u^{\langle N\rangle}_{k,b}=\beta_NR_{b-1,N}\big(\lambda_k^{[N]}\big)$,  see \eqref{eq:kcomponentelefteigenI}.
Hence, recall iii) in Proposition \ref{pro:interlacing}, the Christoffel numbers are positive if and only if 
	\begin{align}\label{eq:pairing}
	&\beta_N 	\xi^{-1} \begin{bNiceMatrix}
		R_{0,N}	 \\\Vdots \\		R_{q-1,N}
		\end{bNiceMatrix}, & &	\frac{1}{\beta_N }\ \begin{bNiceMatrix}
		Q_{0,N}	&\Cdots & 		Q_{p-1,N}	
	\end{bNiceMatrix}\nu^{-\top}
	\end{align}
are positive vectors at the points $x=\lambda_k^{[N]}$, $k\in\{1,\dots,N+1\}$.
We will show now that is possible to choose the initial condition matrices $\nu,\xi$ such that this
holds true.

Now, we consider left and right eigenvectors with last entry normalized to $1$
	\begin{align*}
	&\begin{bNiceMatrix}
	\left.	\frac{Q_{0,N}}{Q_{N,N}} \right|_{x=\lambda_k^{[N]}}&\left.\frac{Q_{1,N} }{Q_{N,N}}\right|_{x=\lambda_1^{[N]}}& \Cdots & 1 
	\end{bNiceMatrix}, &
\begin{bNiceMatrix}
\left.	\frac{R_{0,N}}{R_{N,N}}\right|_{x=\lambda_k^{[N]}} \\ 
\left.\frac{R_{1,N}}{R_{N,N}}\right|_{x=\lambda_k^{[N]}}\\
	 \Vdots \\ 1 
\end{bNiceMatrix}.
\end{align*}
	It is important 	to recall that according to Theorem \ref{teo:eigenvectorII} the last entry of any eigenvector is nonzero, i.e. so that 
		we can normalize the last entry to $1$. Despite, 
		this is not the biorthogonal normalization is interesting for our purposes.
	Recall that $Q_{N,N}=\alpha_N ^{-1}P_N$, $R_{N,N}=\beta_N ^{-1}P_N$	and that, according to Theorem \ref{teo:eigenvectorII}, the first eigenvector entries are not zero; i.e.,
	$	\alpha_N \left.	\frac{Q_{0,N}}{P_{N}} \right|_{x=\lambda_k^{[N]}}$, 
	$\beta_N \left.\frac{R_{0,N}}{P_{N}}\right|_{x=\lambda_k^{[N]}}\neq 0$.
		As the last entry is positive the change sign properties described in Theorem~\ref{teo:eigenvectorII} leads to
		\begin{align*}
		\alpha_N 	\left.\frac{Q_{0,N}}{P_{N}} \right|_{x=\lambda_1^{[N]}}&>0, & 
		\alpha_N 	\left.\frac{Q_{0,N}}{P_{N}} \right|_{x=\lambda_2^{[N]}}&<0,& 
	\alpha_N 	\left.\frac{Q_{0,N}}{P_{N}} \right|_{x=\lambda_3^{[N]}} &>0, & \\
	\beta_N 	\left.\frac{R_{0,N}}{P_{N}} \right|_{x=\lambda_1^{[N]}}&>0, & 
	\beta_N 	\left.\frac{R_{0,N}}{P_{N}} \right|_{x=\lambda_2^{[N]}}&<0,& 
	\beta_N 	\left.\frac{R_{0,N}}{P_{N}} \right|_{x=\lambda_3^{[N]}} &>0, & 
\end{align*}
and so on, alternating the sign. Now, as $T$ is oscillatory and the characteristic polynomial $P_{N+1}$ interlaces $P_N$ we have that 
$ \sgn P_N\big(\lambda_{k}^{[N]}\big)=(-1)^{k-1}$ so that 
		\begin{align*}
	\alpha_N 	Q_{0,N}\big(\lambda_k^{[N]}\big),
	\beta_N R_{0,N}\big(\lambda_k^{[N]}\big)&>0, & k&\in\{1,\dots,N+1\}.
\end{align*}

Now, we start using the Darboux transformations. 	Recall that $\hat{T}^{[N,\pm 1]}$ is an oscillatory matrix with characteristic polynomial $P_{N+1}$. Then, a left eigenvector of $T^{[N,+1]}$ for the eigenvalue $\lambda^{[N]}_k $ can be chosen as 
\begin{align*}
\begin{bNiceMatrix}	\left.	\alpha_N \frac{Q_{0,N}}{P_{N}} \right|_{x=\lambda_k^{[N]}}&\alpha_N \left.\frac{Q_{1,N} }{P_{N}}\right|_{x=\lambda_k^{[N]}}& \Cdots & 1 
\end{bNiceMatrix}L_1^{[N]} 
	&= \begin{bNiceMatrix}
	\alpha_N \left. 	\frac{(Q_{0,N}+ L_{1|0} Q_{1,N})}{P_{N}} \right|_{x=\lambda_k^{[N]}}& \Cdots & 1 
	\end{bNiceMatrix} ,
\end{align*}
and a right eigenvector of $T^{[N,-1]}$ for the eigenvalue $\lambda^{[N]}_k $ can be taken 
 as
 \begin{align*}
 U_1^{[N]} 	\begin{bNiceMatrix}	\left.	\beta_N \frac{R_{0,N}}{P_{N}} \right|_{x=\lambda_k^{[N]}}\\\beta_N \left.\frac{R_{1,N} }{P_{N}}\right|_{x=\lambda_k^{[N]}}\\ \Vdots \\ 1 
 	\end{bNiceMatrix}
 	&= \begin{bNiceMatrix}
 		\left. 	\beta_N \frac{(R_{0,N}+ U_{1|0} R_{1,N})}{P_{N}} \right|_{x=\lambda_k^{[N]}}\\\Vdots \\ 1 
 	\end{bNiceMatrix}.
 \end{align*} 
Using again the sign properties of the eigenvectors associated to an oscillatory matrix we get 
\begin{align*}
	\alpha_N 	\left.\frac{\frac{1}{L_{1|0} }Q_{0,N}+ Q_{1,N}}{P_{N}}\right|_{x=\lambda_1^{[N]}}&>0, &
	\alpha_N 		\left.\frac{\frac{1}{L_{1|0} }Q_{0,N}+Q_{1,N}}{P_{N}}\right|_{x=\lambda_2^{[N]}}&<0, &
	\alpha_N \left.\frac{\frac{1}{L_{1|0} }Q_{0,N}+ Q_{1,N}}{P_{N}}\right|_{x=\lambda_3^{[N]}}&>0, \\
	\beta_N 	\left.\frac{\frac{1}{U_{1|0} }R_{0,N}+ R_{1,N}}{P_{N}}\right|_{x=\lambda_1^{[N]}}&>0, &
	\beta_N 	\left.\frac{\frac{1}{U_{1|0} } R_{0,N}+ R_{1,N}}{P_{N}}\right|_{x=\lambda_2^{[N]}}&<0, 
	&\beta_N \left.\frac{\frac{1}{U_{1|0} }R_{0,N}+ R_{1,N}}{P_{N}}\right|_{x=\lambda_3^{[N]}}&>0,
\end{align*}
and recalling the sign of $P_N$ at the zeros of $P_{N+1}$ we get 
\begin{align*}
	\alpha_N \left.\Big(\frac{1}{L_{1|0} }Q_{0,N}+ Q_{1,N}\Big)\right|_{x=\lambda_k^{[N]}}, 
	\beta_N 	\left.\Big(\frac{1}{U_{1|0} }R_{0,N}+ R_{1,N}\Big)\right|_{x=\lambda_k^{[N]}}&>0, & k&\in\{1,\dots,N+1\}.
\end{align*}
Now we consider the matrices $\hat{T}^{[N,\pm 2]}$, both oscillatory matrices with characteristic polynomial $P_{N+1}$. Then, for $T^{[N,+2]}$, a corresponding left eigenvector for the eigenvalue $\lambda^{[N]}_k $ can be chosen as 
\begin{align*}
		\begin{bNiceMatrix}	\left.	\alpha_N \frac{Q_{0,N}}{P_{N}} \right|_{x=\lambda_k^{[N]}}&\alpha_N \left.\frac{Q_{1,N} }{P_{N}}\right|_{x=\lambda_k^{[N]}}& \Cdots & 1 
		\end{bNiceMatrix}L_1^{[N]} L_2^{[N]} 
	= \begin{bNiceMatrix}
	\alpha_N \left.	\frac{Q_{0,N}+(L_{1|0}+ L_{2|0})Q_{1,N}+ L_{1|1}L_{2|0} Q_{2,N}}{P_{N}}\right|_{x=\lambda_k^{[N]}} & \Cdots & 1 
	\end{bNiceMatrix},
\end{align*}
and for $T^{[N,-2]}$ a corresponding right eigenvector for the eigenvalue $\lambda^{[N]}_k $ can be taken as 
\begin{align*}
U_2U_1	\begin{bNiceMatrix}	\left.	\beta_N \frac{R_{0,N}}{P_{N}} \right|_{x=\lambda_k^{[N]}}\\\beta_N \left.\frac{R_{1,N} }{P_{N}}\right|_{x=\lambda_k^{[N]}}\\ \Vdots \\ 1 
	\end{bNiceMatrix}
	= \begin{bNiceMatrix}
		\beta_N \left.	\frac{R_{0,N}+(U_{1|0}+ U_{2|0})R_{1,N}+ U_{1|1}U_{2|0} R_{2,N}}{P_{N}}\right|_{x=\lambda_k^{[N]}} \\ \Vdots \\ 1 
	\end{bNiceMatrix}.
\end{align*}
 Hence, using the sign properties of the eigenvectors associated to an oscillatory matrix we get 
\begin{align*}
	\alpha_N 	\left.\frac{\frac{1}{L_{1|1}L_{2|0} }Q_{0,N}+\frac{L_{1|0}+ L_{2|0}}{ L_{1|1}L_{2|0} }Q_{1,N}+Q_{2,N}}{P_{N}}\right|_{x=\lambda_1^{[N]}} &>0, &
	\alpha_N 	\left.\frac{\frac{1}{L_{1|1}L_{2|0} }Q_{0,N}+\frac{L_{1|0}+ L_{2|0}}{ L_{1|1}L_{2|0} }Q_{1,N}+Q_{2,N}}{P_{N}}\right|_{x=\lambda_2^{[N]}} &<0, 
	\\
	\beta_N 	\left.\frac{\frac{1}{U_{1|1}U_{2|0} }R_{0,N}+\frac{U_{1|0}+ U_{2|0}}{ U_{1|1}U_{2|0} }R_{1,N}+R_{2,N}}{P_{N}} \right|_{x=\lambda_1^{[N]}}&>0, &
	\beta_N 	\left.\frac{\frac{1}{U_{1|1}U_{2|0}} R_{0,N}+\frac{U_{1|0}+ U_{2|0}}{ U_{1|1}U_{2|0} }R_{1,N}+R_{2,N}}{P_{N}} \right|_{x=\lambda_2^{[N]}}&<0, 
\end{align*}
and so on, alternating the sign. Recalling the sign of the characteristic polynomial $P_N$ at the eigenvalues of $T^{[N]}$ we get
\begin{align*}
	\alpha_N 	\left.\frac{\frac{1}{L_{1|1}L_{2|0} }Q_{0,N}+\frac{L_{1|0}+ L_{2|0}}{ L_{1|1}L_{2|0} }Q_{1,N}+Q_{2,N}}{P_{N}}\right|_{x=\lambda_1^{[N]}} &>0, &
	\alpha_N 	\left.\frac{\frac{1}{L_{1|1}L_{2|0} }Q_{0,N}+\frac{L_{1|0}+ L_{2|0}}{ L_{1|1}L_{2|0} }Q_{1,N}+Q_{2,N}}{P_{N}}\right|_{x=\lambda_2^{[N]}} &<0, 
	\\
	\beta_N 	\left.\frac{\frac{1}{U_{1|1}U_{2|0} }R_{0,N}+\frac{U_{1|0}+ U_{2|0}}{ U_{1|1}U_{2|0} }R_{1,N}+R_{2,N}}{P_{N}} \right|_{x=\lambda_1^{[N]}}&>0, &
	\beta_N 	\left.\frac{\frac{1}{U_{1|1}U_{2|0}} R_{0,N}+\frac{U_{1|0}+ U_{2|0}}{ U_{1|1}U_{2|0} }R_{1,N}+R_{2,N}}{P_{N}} \right|_{x=\lambda_2^{[N]}}&<0 .
\end{align*}
Hence, we obtain
\begin{align*}
	\alpha_N \left.\Big(\frac{1}{L_{1|1}L_{2|0} }Q_{0,N}+\frac{L_{1|0}+ L_{2|0}}{ L_{1|1}L_{2|0} }Q_{1,N}+Q_{2,N}\Big)\right|_{x=\lambda_k^{[N]}} &>0, \\
	\beta_N 	\left.\Big(
	\frac{1}{U_{1|1}U_{2|0} }R_{0,N}+\frac{U_{1|0}+ U_{2|0}}{ U_{1|1}U_{2|0} }R_{1,N}+R_{2,N}
	\Big)\right
	|_{x=\lambda_k^{[N]}}&>0, 
\end{align*}
for $k\in\{1,\dots,N+1\}$.

Consequently, after repeating this process up to $T^{[N,+(p-1)]}$ and $T^{[N,-(q-1)]}$ we find that
	\begin{align*}
	&\beta_N 	\Upsilon \begin{bNiceMatrix}
		R_{0,N}		\\\Vdots \\		R_{q-1,N}		\end{bNiceMatrix}, &
	&\alpha_N \begin{bNiceMatrix}
	Q_{0,N}	&\Cdots & 		Q_{p-1,N}	 
\end{bNiceMatrix}\Lambda ,
\end{align*}
are positive vectors at the points $x=\lambda^{[N]}_k$, $k\in\{1,\dots,N\}$. Therefore, if the initial condition matrices are chosen as indicated in \eqref{eq:tunnig_initial_conditions}
we get the result.
	\end{proof}

			\section{Resolvent, second kind polynomials and Weyl functions}

			From here on we assume that $N\geqslant \max(p,q)$.	
\begin{defi}
		Given $r\in\N$, we write  $\big\{e^{[r]}_1,\dots,e^{[r]}_r\big \}$ for  the canonical basis of $\R^r$ and consider the  $r\times(N+1)$ matrix
			$	E_{[r]}\coloneq\begin{bNiceMatrix}[small]
				I_r & 0_{r\times(N+1-r)}
			\end{bNiceMatrix}$.
Then,  we  introduce the vectors $e^\nu_a,e^\xi_b\in\R^{N+1}$ with
		\begin{align*}
		e^\nu_a&\coloneq E_{[p]}^\top\nu^{-\top}e^{[p]}_a, & \big(e^\xi_b\big)^\top&\coloneq \big(e^{[q]}_b\big)^\top \xi^{-1}E_{[q]}.
		\end{align*}
\end{defi}		 
\begin{lemma}\label{lem:truco_o_trato}
	For the matrices $\mathscr U$ and $\mathscr W$ (introduced in v) of Proposition \ref{pro:UW}) we find
	\begin{align*}
		(e_b^\xi)^\top \mathscr U&=\begin{bNiceMatrix}
				\rho^{[N]}_{1,b} &\Cdots &	\rho^{[N]}_{N+1,b}
		\end{bNiceMatrix}, &
		\mathscr W e_a^\nu&= \begin{bNiceMatrix}
			\mu^{[N]}_{1,a}\\\Vdots\\ \mu^{[N]}_{N+1,a}
		\end{bNiceMatrix}, & a&\in\{1,\dots,p\}, & b&\in\{1,\dots,q\}.
	\end{align*}
\end{lemma}

\begin{rem}
	In matrix form, the above Lemma \ref{lem:truco_o_trato} reads
	\begin{align*}
		\begin{bNiceMatrix}
		\rho^{[N]}_{1,1} &\Cdots&\rho^{[N]}_{N+1,1} \\
		\Vdots & & \Vdots \\
		\rho^{[N]}_{1,q} &\Cdots&\rho^{[N]}_{N+1,q}
	\end{bNiceMatrix}	&=\xi^{-1} E_{[q]}\mathscr U,&	\begin{bNiceMatrix}
		\mu^{[N]}_{1,1} &\Cdots&\mu^{[N]}_{1,p} \\
		\Vdots & & \Vdots \\
		\mu^{[N]}_{N+1,1} &\Cdots&\mu^{[N]}_{N+1,p}
	\end{bNiceMatrix}&=\mathscr W E_{[p]}^\top\nu^{-\top}.
\end{align*}
Also observe that \eqref{eq:rho_mu_xi_nu} follows from these relations and two facts: $\mathscr U\mathscr W=I_N$ and $E_{[q]}E_{[p]}^\top=I_{q,p}$.
\end{rem}

For the following we need of the adjugate matrix $\operatorname{adj} A$ of a matrix $A$; i.e. of the transpose of the matrix of cofactors (c.f \cite{Horn-Johnson}).
\begin{defi}[Second kind polynomials]
	The second kind characteristic polynomials are given by
	\begin{align*}
		P^{(b,a)}_{N+1}(x)&\coloneq (e_b^\xi)^\top \operatorname{adj}(xI_{N+1}-T^{[N]})e^\nu_a, &a&\in\{1,\dots,p\}, & b&\in\{1,\dots,q\}.
	\end{align*}
\end{defi}
\begin{pro}\label{pro:second_kind_characteristic_I}
	For the second kind characteristic polynomials we find
	\begin{align*}
		P^{(b,a)}_{N+1}(x)&=\sum_{k=1}^{N+1}\rho_{k,b}^{[N]}\mu_{k,a}^{[N]}\pi^{[N]}_k(x), & \pi^{[N]}_k(x)&\coloneq\prod_{\mathclap{\substack{l\in\{1,\dots,N+1\}\\l\neq k}}}\big(x-\lambda^{[N]}_l\big).
	\end{align*}
\end{pro}
\begin{proof}
	It follows from $\operatorname{adj} (xI_{N+1}-T^{[N]})=\operatorname{adj} (\mathscr U(xI_{N+1}-D)\mathscr W)= \mathscr U\operatorname{adj}(xI_{N+1}-D)\mathscr W$.
\end{proof}

\begin{pro}
	The second kind characteristic polynomials are the second kind polynomials of the characteristic polynomial; i.e.,
	\begin{align*}
		P^{(b,a)}_{N+1}(z)&=
		\int \frac{P_{N+1}(z)-P_{N+1}(x)}{z-x}\d\psi_{b,a}^{[N]}(x)\\&=\alpha_{N+1}\int \frac{\det (A_{N+1}(z))-\det (A_{N+1}(x))}{z-x}\d\psi_{b,a}^{[N]}(x)\\&=\beta_{N+1}\int \frac{\det (B_{N+1}(z))-\det( B_{N+1}(x))}{z-x}\d\psi_{b,a}^{[N]}(x).
	\end{align*}
\end{pro}
\begin{proof}
We have
\begin{align*}
	\int \frac{P_{N+1}(z)-P_{N+1}(x)}{z-x}\d\psi_{b,a}^{[N]}(x)=\sum_{k=1}^{N+1}\rho_{k,b}^{[N]}\mu_{k,a}^{[N]}\int\delta\big(x-\lambda_k^{[N]}\big)\frac{P_{N+1}(z)-P_{N+1}(x)}{z-x},
\end{align*}
but
\begin{align*}
	\int\delta\big(x-\lambda_k^{[N]}\big)\frac{P_{N+1}(z)-P_{N+1}(x)}{z-x}=
	\frac{P_{N+1}(z)-P_{N+1}\big(\lambda_k^{[N]}\big)}{z-\lambda_k^{[N]}}
			=\frac{P_{N+1}(z)}{z-\lambda_k^{[N]}}=\pi^{[N]}_k(z),
\end{align*}
and using Proposition \ref{pro:second_kind_characteristic_I} we obtain the first result. For the second we use Theorem \ref{teo:determinants_mixed}.
\end{proof}
\
\begin{rem}
The second kind characteristic polynomial matrix is
	\begin{align*}
		P^{(1)}_{N+1}&\coloneq \xi^{-1} E_{[q]}\operatorname{adj}(xI_{N+1}-T^{[N]}) E_{[p]}^\top\nu^{-\top}=\sum_{k=1}^{N+1}\pi^{[N]}_k(x)	\begin{bNiceMatrix}
			\rho^{[N]}_{k,1} \\\Vdots \\	\rho^{[N]}_{k,q}
		\end{bNiceMatrix}\begin{bNiceMatrix}
			\mu^{[N]}_{k,1} &\Cdots & 	\mu^{[N]}_{k,p}
		\end{bNiceMatrix}\\&=\int \frac{P_{N+1}(z)-P_{N+1}(x)}{z-x}\d\Psi^{[N]}(x),
	\end{align*}
is a $q\times p$ matrix of polynomials whose entries are the polynomials of the second kind: $(P^{(1)}_{N+1})_{b,a}=	P^{(b,a)}_{N+1}$.
\end{rem}
\begin{pro}
	If $T$ has a PBF and \eqref{eq:tunnig_initial_conditions} is satisfied then $\deg P_{N+1}^{(b,a)}=N$.
\end{pro}
\begin{proof}
	The choice \eqref{eq:tunnig_initial_conditions} ensures that the entries of the vectors $e^\nu_a$ and $e^\xi_b$ are positive. The PBF of $T$ also ensures that all the Christoffel numbers are positive. Then, the definition of the second kind polynomials through the adjugate matrix leads to the degree $N$ of these polynomials.
\end{proof}

The moments of the $pq$ discrete measures $\d\psi_{b,a}^{[N]}$ are linked to the components of the powers of $T^{[N]}$:
\begin{pro}[Discrete moments]\label{pro:eq:moments_discrete}
	For the discrete moments we have
	\begin{align}\label{eq:moments_discrete}
	\int x^n\d\psi^{[N]} _{b,a}(x) &=\sum_{k=1}^{N+1}\rho_{k,b}^{[N]}\mu_{k,a}^{[N]}\big( \lambda_k^{[N]}\big)^n
		=(e_b^\xi)^\top \big(T^{[N]}\big)^ne^\nu_a, & a&\in\{1,\dots, p\},& b&\in\{1,\dots,q\}.
	\end{align}
\end{pro}
	\begin{proof}
		We have that $	(e_b^{\xi})^\top \big(T^{[N]}\big)^ne^\nu_a=	(e_b^\xi)^\top \mathscr U D^n \mathscr We^\nu_a$ so that
		\begin{align*}
			(e_b^{\xi})^\top\big(T^{[N]}\big)^ne^\nu_a=\begin{bNiceMatrix}
				\rho^{[N]}_{1,b} &\Cdots &	\rho^{[N]}_{N+1,b}
			\end{bNiceMatrix}D^n
			\begin{bNiceMatrix}
				\mu^{[N]}_{1,a}\\\Vdots\\ \mu^{[N]}_{N+1,a}
			\end{bNiceMatrix},
		\end{align*}
		and the result follows.
	\end{proof}
	
	\begin{rem}
		In matrix form we can write 
		\begin{align*}
			\int x^n\d\Psi^{[N]}(x)=\xi^{-1} E_{[q]}\big(T^{[N]}\big)^nE_{[p]}^\top\nu^{-\top}.
		\end{align*}
		\end{rem}
\begin{defi}[Resolvent]
	The resolvent matrix $ R^{[N]}(z)$ of the leading principal submatrix $ T^{[N]}$ is
	\begin{align*}
		R^{[N]}(z)\coloneq \big(z I_{N+1}- T^{[N]}\big)^{-1}=\frac{\operatorname{adj}\big(z I_{N+1}- T^{[N]}\big)}{\det(z I_{N+1}-T^{[N]})}.
	\end{align*}
\end{defi}
	\begin{lemma}We have	
	\begin{align}\label{eq:spectral_resolvent}
		R^{[N]}(z)
		& =\mathscr U(zI_{N+1}-D)^{-1}\mathscr W. \end{align}
	\end{lemma}
\begin{proof}
	It follows immediately from the spectral decomposition of the matrix $ T^{[N]}$.
\end{proof}
	\begin{defi}[Weyl's functions]
		The Weyl functions are
		\begin{align*}
			S_{b,a}^{[N]}&\coloneq (e_b^\xi)^\top R^{[N]}e^\nu_a, & a&\in\{1,\dots, p\}, & b&\in\{1,\dots,q\}.
		\end{align*}
	\end{defi}
	
	\begin{pro}\label{pro:Weyl_functions}
		The Weyl functions can be expressed as follows
		\begin{align*}
			S_{b,a}^{[N]}(z)&= \frac{P^{(b,a)}_{N+1}(z)}{P_{N+1}(z)}=\sum_{k=1}^{N+1}\frac{\rho^{[N]}_{k,b}\mu^{[N]}_{k,a}}{z-\lambda^{[N]}_k}=
			\int\frac{\d\psi^{[N]}_{b,a}(x)}{z-x}
			, & a&\in\{1,\dots, p\}, & b&\in\{1,\dots,q\}.
		\end{align*}
	\end{pro}
	\begin{proof}
		The first equalities follow from adjugate expressions.
		The second expressions can be deduced from \eqref{eq:spectral_resolvent}. Indeed, recalling Lemma \ref{lem:truco_o_trato} we get that
		the Weyl functions are
		\begin{align*}
			S_{b,a}^{[N]}(z)&= \begin{bNiceMatrix}
				\rho^{[N]}_{1,b}&\Cdots& \rho^{[N]}_{N+1,b}
			\end{bNiceMatrix}(zI_{N+1}-D)^{-1}
			\begin{bNiceMatrix}
				\mu^{[N]}_{1,a}\\\Vdots\\ \mu^{[N]}_{N+1,a}
			\end{bNiceMatrix}=\sum_{k=1}^{N+1}\frac{\rho^{[N]}_{k,b}\mu^{[N]}_{k,a}}{z-\lambda^{[N]}_{k}}.
		\end{align*}
	\end{proof}

\begin{rem}
	For the $q\times p$ matrix of Weyl functions
$	S^{[N]}=\begin{bNiceMatrix}[small]
			S^{[N]}_{1,1} &\Cdots& 	S^{[N]}_{1,p} \\
			\Vdots & &\Vdots\\
			S^{[N]}_{q,1} &\Cdots& 	S^{[N]}_{q,p} 
		\end{bNiceMatrix}\coloneq \xi^{-1}E_{[q]}R_z^{[N]}E^\top_{[p]}\nu^{-1}$,
we can write
	\begin{align*}
		S^{[N]}(z)&= \frac{P^{(1)}_{N+1}(z)}{P_{N+1}(z)}=\sum_{k=1}^{N+1}\frac{1}{z-\lambda^{[N]}_k}	\begin{bNiceMatrix}
			\rho^{[N]}_{k,1} \\\Vdots \\	\rho^{[N]}_{k,q}
		\end{bNiceMatrix}\begin{bNiceMatrix}
			\mu^{[N]}_{k,1} &\Cdots & 	\mu^{[N]}_{k,p}
		\end{bNiceMatrix}		
		=\int\frac{	\d\Psi^{[N]}(x)}{z-x}.
	\end{align*}

\end{rem}

\begin{pro}
	If $T$ has a PBF and \eqref{eq:tunnig_initial_conditions} is satisfied then $P_{N+1}^{(b,a)}$ is interlaced by $P_{N+1}$.
\end{pro}

\begin{proof}
	Notice that if $T$ has a PBF all the singularities of the Weyl functions are simple poles with positive residues. Consequently, each of the second kind polynomials $P^{(b,a)}_{N+1}$ is interlaced by the characteristic polynomial $P_{N+1}$.
\end{proof}

We now connect these constructions with the polynomials used in discussion of the Hermite--Padé problem in Proposition \ref{pro:matrix_Hermite_Padé}.
As usual, $\{e_n\}_{n=1}^{N+1}\subset R^{N+1}$ is the canonical basis.
\begin{pro}[Vectorial polynomials of the second type]
	For $n\in\{1,\dots,N+1\}$ we find
	\begin{align*} 
\int \frac{\d\Psi^{[N]}(x)	}{z-x}\begin{bNiceMatrix}
	A^{(1)}_{n-1}(x)\\\Vdots \\A^{(p)}_{n-1}(x)
\end{bNiceMatrix}&=\xi^{-1}E_{[q]}R^{[N]}(z) e_n, \\
\int \begin{bNiceMatrix}
	B^{(1)}_{n-1}(x)& \Cdots &B^{(q)}_{n-1}(x)
\end{bNiceMatrix}\frac{\d\Psi^{[N]}(x)	}{z-x}&=e_n^\top R^{[N]}(z) E_{[p]}^\top\nu^{-\top}, 
\end{align*}	
and entrywise
	\begin{align*} 
\sum_{a=1}^p	\int \frac{\d\psi_{b,a}^{[N]}(x)	}{z-x}
	A^{(a)}_{n-1}(x)&=\big(e^\xi_b\big)^\top R^{[N]}(z) e_n, &
\sum_{b=1}^q	\int 
		B^{(b)}_{n-1}(x)\frac{\d\psi_{b,a}^{[N]}(x)	}{z-x}&=e_n^\top R^{[N]}(z) e_a^\nu.
\end{align*}	
\end{pro}
\begin{proof}
	From \eqref{eq:discrete_linear_form_typeI} we get
		\begin{align*} 
		\int \frac{\d\Psi^{[N]}(x)	}{z-x}\begin{bNiceMatrix}
			A^{(1)}_{n-1}(x)\\\Vdots \\A^{(p)}_{n-1}(x)
		\end{bNiceMatrix}&=\sum_{k=1}^{N+1}\begin{bNiceMatrix}
		\rho_{k,1}
\\
\Vdots\\
\rho_{k,q}	\end{bNiceMatrix}\frac{w^{\langle N\rangle}_{k,n}}{z-\lambda_k^{[N]}}
	\end{align*}	
	and \eqref{eq:kcomponentelefteigenII} implies
			\begin{align*} 
		\int \frac{\d\Psi^{[N]}(x)	}{z-x}\begin{bNiceMatrix}
			A^{(1)}_{n-1}(x)\\\Vdots \\A^{(p)}_{n-1}(x)
		\end{bNiceMatrix}&=\xi^{-1}\sum_{k=1}^{N+1}\begin{bNiceMatrix}
			u_{k,1}^{\langle N\rangle}
			\\
			\Vdots\\
			u_{k,q}^{\langle N\rangle}	\end{bNiceMatrix}\frac{w^{\langle N\rangle}_{k,n}}{z-\lambda_k^{[N]}}
		=\xi^{-1} E_{[q]}\sum_{k=1}^{N+1}\begin{bNiceMatrix}
			u_{k,1}^{\langle N\rangle}
			\\
			\Vdotsfor{2}\\\\
			u_{k,N+1}^{\langle N\rangle}\end{bNiceMatrix}\frac{1}{z-\lambda_k^{[N]}}\begin{bNiceMatrix}
			w_{k,1}^{\langle N\rangle}
			&
			\Cdots&
			w_{k,N+1}^{\langle N\rangle}\end{bNiceMatrix}e_n,
	\end{align*}	
and the result follows. Now, proceeding similarly and using \eqref{eq:discrete_linear_form_typeII} and \eqref{eq:kcomponentelefteigenII} we obtain the second relation.
\end{proof}

\section{Spectral Favard theorem}


As the submatrices $T^{[N]}$ are oscillatory, we know that $P_{N+1}(x)$ strictly interlaces $P_N(x)$ so that the positive sequence  $\{\lambda_1^{[N]}\}_{N=1}^\infty$ is a strictly increasing sequence and 
  $\{\lambda_{N+1}^{[N]}\}_{N=1}^\infty$ is a strictly decreasing sequence.
As well, for bounded operators, $\|T\|_\infty <\infty$, we have $\|T^{[N]}\|_\infty <\|T\|_\infty <\infty$. Therefore, there exists the limits
$\zeta\coloneq \lim_{N\to\infty }\lambda_{N+1}^{[N]}\geqslant 0$ and $\eta\coloneq \lim_{N\to\infty }\lambda_1^{[N]}\leqslant \|T\|_\infty$. 
Following \cite{Chihara, Ismail} we call $\Delta\coloneq [\zeta,\eta]\subseteq [0,\|T\|_\infty]$ the true interval of orthogonality, that is the smallest interval containing all zeros of the characteristic polynomials $P_n$, i.e. the eigenvalues of the leading principal submatrices of $T$.

\begin{teo}[Favard spectral representation]\label{theorem:spectral_representation_bis}
	Let us assume that
	\begin{enumerate}
\item The banded  matrix $T$ is bounded and there exist $s\geqslant 0$ such that $T+sI$ has a PBF.
		\item The sequences $\big\{A^{(1)}_n,\dots,A^{(p)}_n\big\}_{n=0}^\infty, \big\{B^{(1)}_n,\dots,B^{(q)}_n\big\}_{n=0}^\infty$ 
		of recursion polynomials are determined by the initial condition matrices $\nu$ and $\xi$,respectively, such that $\nu^{-\top}=\Lambda\mathcal A$, $\xi^{-1}=\mathcal B \Upsilon$, and $\mathcal A\in \R^{p\times p}$ is a nonnegative upper unitriangular matrices and $\mathcal B\in \R^{q\times q}$ is a nonnegative lower unitriangular matrix.
	\end{enumerate}
	Then, there exists $pq$ non decreasing positive functions $\psi_{b,a}$, $a\in\{1,\dots,p\}$ and $b\in\{1,\dots,q\}$
	and corresponding positive Lebesgue--Stieltjes measures $\d\psi_{b,a}$ with compact support $\Delta$ such that the following biorthogonality holds
	\begin{align*}
	\sum_{a=1}^p\sum_{b=1}^q	\int_{\Delta} B^{(b)}_l(x)\d\psi_{b,a}(x) A^{(a)}_{k}(x)&= \delta_{k,l}, &k,l&\in\N_0.
	\end{align*}
\end{teo}
\begin{proof}
		The shift in the matrix $T\to T+sI$ only shifts by $s$ the eigenvalues of the truncations $T^{[N]}$, so that they are positive, and the dependent variable of the recursion polynomials, but do not alter the interlacing properties of the polynomials and the positivity of the corresponding Christoffel numbers. 
From Theorem \ref{teo:positivity} we know that the sequences $\big\{\psi_{a,b}^{[N]}\big\}_{N=0}^\infty$, $a\in\{1,\dots,p\}$, $b\in\{1,\dots,q\}$ given in Definition \ref{def:Stieltjes} are positive. Moreover, Proposition \ref{pro:boudness_Christoffel} implies that they are uniformly bounded and nondecreasing.
	Consequently, following Helly's results, see \cite[\S II]{Chihara} there exist subsequences that converge when $N\to\infty$ to positive nondecreasing functions $\psi_{b,a}$ with support on $\Delta$ and that the discrete biorthogonal relations lead to the stated biorthogonal properties.
\end{proof}


\begin{coro}[Mixed multiple orthogonal relations]\label{cor:mor}
	In the conditions of Theorem \ref{theorem:spectral_representation_bis}, the mixed multiple orthogonal relations are fulfilled
	\begin{align*}
		\sum_{a=1}^p \int_{\Delta }x^n\d \psi_{b,a}(x)A^{(a)}_{m}(x)&=0, & 
		n&\in\big\{0,\dots,\deg B^{(b)}_{m-1}\big\},&	b&\in\{1,\dots,q\},\\
		\sum_{b=1}^q \int_{\Delta }	B^{(b)}_{m}(x)\d \psi_{b,a}(x)x^n &=0, &				n&\in\big\{0,\dots,\deg A^{(a)}_{m-1}\big\}, & a&\in\{1,\dots,p\}.
	\end{align*}
%
\end{coro}


\begin{pro}[Spectral representation of moments and Stieltjes--Markov functions]\label{pro:spectral_moments_Weyl}
In the conditions of Theorem \ref{theorem:spectral_representation_bis} and in terms of the spectral functions $\psi_{b,a}$, $a\in\{1,\dots,p\}$, $b\in\{1,\dots,q\}$ we find the following relations between entries of powers or the resolvent of the banded matrix and moments or the Cauchy transform of the measures, respectively:
	\begin{align*}
	(	e_b^\xi)^\top T^n e^\nu_a&=\int_\Delta x^n\d\psi_{b,a}(x),& (e_b^\xi)^\top (z I- T)^{-1}e^\nu_a&=\int_\Delta\frac{\d\psi_{b,a}(x)}{z-x}\eqcolon\hat \psi_{b,a}(z).
	\end{align*}
\end{pro}
\begin{proof}
	Propositions \ref{pro:eq:moments_discrete} and \ref{pro:Weyl_functions} and	Helly's second theorem leads to the spectral representation for the moments and Stieltjes--Markov functions $\hat \psi_{b,a}(z)$ of $T$.
\end{proof}

\begin{rem}
	In terms of 
$\Psi=\begin{bNiceMatrix}
		\psi_{1,1}&\Cdots &\psi_{1,p}\\
		\Vdots & & \Vdots\\
		\psi_{q,1}&\Cdots &\psi_{q,p}
	\end{bNiceMatrix}$
we have
	\begin{align*}
	\xi^{-1} E_{[q]}T^n E_{[p]}^\top \nu^{-\top}&=\int_\Delta x^n\d\Psi(x), &	\xi^{-1} E_{[q]} (z I- T)^{-1}E_{[p]}^\top \nu^{-\top}&=\int_\Delta\frac{\d\Psi(x)}{z-x}.
\end{align*}
Now, for $r\in\N$, we use the semi-infinite matrix
\begin{align*}
	E_{[r]}\coloneq\left[\begin{NiceArray}{cccc|cccccc}
	1&0&\Cdots&0 & 0&\Cdots&&& &\\
	0&\Ddots^{\text{$r$ times}}&\Ddots& \Vdots &\Vdots^{\text{$r$ times}}&&&&\\
	\Vdots&\Ddots&\Ddots& 0 &&&&&\\
	0&\Cdots&0& 1 &0&\Cdots&&&
	\end{NiceArray}\right].
\end{align*}
\end{rem}

\begin{pro}[Normal convergence of Weyl functions]
	Given the conditions of Theorem \ref{theorem:spectral_representation_bis}, the Weyl functions in Proposition \ref{pro:Weyl_functions} converge uniformly in compact subsets of $\bar\C\setminus \Delta$ to the Stieltjes--Markov functions, i.e.,
	\begin{align*}
		S^{[N]}_{b,a}(z)= \frac{P^{(b,a)}_{N+1}(z)}{P_{N+1}(z)}&\xrightrightarrows [N\to\infty]{}\hat \psi_{b,a}(z), & a&\in\{1,\dots,p\}, & b&\in\{1,\dots,q\}.
	\end{align*}
\end{pro}
\begin{proof}
Notice the uniform boundedness  in $N$  in compact subsetsof $\bar\C\setminus \Delta$ of the Weyl functions $S^{[N]}_{b,a}$ for each pair $a,b$. Then, Vitali convergence theorem see \cite[Theorem 6.2.8]{Simon2}
leads to the result.\end{proof}

\begin{rem}
	Despite the positivity of Christoffel numbers described in Theorem \ref{teo:positivity} we only have the bound proved in Proposition \ref{pro:boudness_Christoffel}. Therefore, we know that the functions $\psi^{[N]}_{b,a}$ given in Definition \ref{def:Stieltjes} that are right continuous, of bounded variation, increasing and positive are also uniformly bounded. Therefore, Helly's theorem can be applied to the large $N$ limit.
	
	However, this is not applicable to each family of Christoffel numbers separately. That is,
	\begin{align*}
		\varphi^{[N]}_{b}&\coloneq \begin{cases}
			0, & x<\lambda^{[N]}_{N+1},\\[2pt]
			\rho^{[N]}_{1,b}+\cdots+\rho^{[N]}_{k,b}, & \lambda^{[N]}_{k+1}\leqslant x< \lambda^{[N]}_{k}, \quad k\in\{1,\dots,N\},\\[2pt]
			\rho^{[N]}_{1,b}+\cdots+\rho^{[N]}_{N+1,b},
			& x \geqslant \lambda^{[N]}_{1},
		\end{cases}\\
		\tilde \varphi^{[N]}_{a}&\coloneq \begin{cases}
			0, & x<\lambda^{[N]}_{N+1},\\[2pt]
			\mu^{[N]}_{1,a}+\cdots+\mu^{[N]}_{k,a}, & \lambda^{[N]}_{k+1}\leqslant x< \lambda^{[N]}_{k}, \quad k\in\{1,\dots,N\},\\[2pt]
			\mu^{[N]}_{1,a}+\cdots+\mu^{[N]}_{N+1,a} ,
			& x \geqslant \lambda^{[N]}_{1},
		\end{cases}
	\end{align*}
	are right continuous, of bounded variation, increasing and positive. But, in principle, they might be not bounded and therefore Helly's result may not be applicable. 
	
	Thus, to get measures from these  functions we need to ensure the existence of bounds as follows
$\rho^{[N]}_{1,b}+\cdots+\rho^{[N]}_{N+1,b}\leqslant R_{b}$ and $\mu^{[N]}_{1,a}+\cdots+\mu^{[N]}_{N+1,a} \leqslant M_{a}$.
	For such situation, the large limit will lead to the existence of spectral measures $\d\varphi_a$, $a\in {1,\dots,p}$ and $\d\tilde \varphi_b$, $b\in\{1,\dots,q\}$. If these measures are absolutely continuous w.r.t. the measure $\d\mu$, with Radon--Nikodym derivatives the weights $w_a$ and $\tilde w_b$, respectively, we could write $\d\varphi_a=w_a\d\mu$ and $\d\tilde\varphi_b=\tilde w_b\d\mu$. A natural conjecture, that we have not yet proven, is that in this situation $\d\psi_{b,a}=w_a\tilde w_b\d\mu$. This rank one simplification is assumed in a large number of papers dealing with mixed multiple orthogonality.
\end{rem}

\section{Mixed multiple Gaussian quadrature and degrees of precision}

Gaussian quadrature formulas are an important tool in the theory of orthogonal polynomials and its applications to approximation theory, see for example \cite{Chihara,Ismail}. Its extension to non-mixed multiple orthogonal polynomials was discussed in \cite{Borges,Ulises-Illan-Guillermo,Coussement-VanAssche,PBF_1}, degrees of precision were presented in \cite{PBF_1}. Now, we give its extension to the mixed multiple orthogonal situation. Notice that for $p=q$ we are dealing with standard matrix orthogonal polynomials and such quadrature formulas have been discussed for this situation,
see the excellent review \cite{Duran_Grunbaum} and references therein cited.

Let us assume  that $T$ has a PBF, and that the conditions of Theorem \ref{theorem:spectral_representation_bis} hold, and introduce:
\begin{defi}
	The degrees of precision or orders $d_{b,a}(N)$, $a\in\{1,\dots,p\}$, $b\in\{1,\dots,q\}$, are the largest natural numbers such that
	\begin{align*}
		\big(e_b^\xi\big)^\top T^n e^\nu_a &=	(e_b^\xi)^\top \big(T^{[N]}\big)^n e^\nu_a, & 0&\leqslant n\leqslant d_{b,a}(N), & a&\in\{1,\dots,p\},& b&\in\{1,\dots,q\} .
	\end{align*}
\end{defi}

\begin{pro}\label{pro:degrees_precision}
	In terms of the recursion polynomial degrees, see Proposition \ref{pro:degrees},	the degrees of precision~are
	\begin{align*}
		d_{b,a}(N)&= \deg A^{(a)}_N+\deg B^{(b)}_N+1=\left\lceil\frac{ N+2-a}{p}\right\rceil+\left\lceil\frac{ N+2-b}{q}\right\rceil-1,& a&\in\{1,\dots,p\}, & b&\in\{1,\dots,q\}.
	\end{align*}
\end{pro}
	\begin{proof}
The degrees of precision give us the last power so that the only entries in the principal truncations are involved. 

Notice that $e_b^\xi$ and $e^\nu_a$ are positive vectors involving the first $b$ or $a$ entries, respectively. Therefore, in the computation of $\big(e_b^\xi\big)^\top T^n e^\nu_a $, what we need to analyze is the $(i_0,i_n)$, for $i_0\in\{0,\dots,b-1\}$ and $i_n\in\{0,\dots,a-1\}$, entry of the $n$-th power of the banded matrix $T$ and check whether it involves nonzero factors $T_{i,j}$ with $i>N$ and/or $j>N$. Notice that these entries are sums of products of the form
$T_{i_0,i_1}T_{i_1,i_2}\cdots T_{i_{n-2},i_{n-1}}T_{i_{n-1},i_n}$ where all the factors in the product are chosen to be positive entries of the banded matrix $T$. Hence, we can identify each of these positive contributions to our $(i_0,i_n)$ entry as the string $(i_0,i_1,i_2,\dots,i_{n-1},i_n)$, we refer to each element in this string as terms. For each $(b,a)$ the strings that involve more terms are those that correspond to $i_0=b-1$ and $i_n=a-1$. Thus, these maximal strings are the ones that could achieve the appearance of non desired nonzero factors $T_{i,j}$ with $i>N$ and/or $j>N$ the sooner.
Hence, for given $a$ and $b$, we need to find those maximal strings that achieve this target in the fastest way, i.e. by jumping in the string elements by $q$ by $q$, and stopping just before the first zero factor. Consequently, we are dealing with the following strings 
\begin{multline*}\small
	(b-1,b-1+q,b-1+2q,	\dots, b-1+(r-1)q,b-1+rq,\\b-1+rq-p,b-1+rq-2p,\cdots,b-1+rq-(s-1)p,a-1)
\end{multline*} 
with $r,s\in\N$ and such $b-1+rq$ is at much $N$, so that 
$r=\left\lceil\frac{ N+2-b}{q}\right\rceil-1$.
Now, $s$ is such we get down to $a-1$ as fast as possible; i.e., by jumps in the matrix indexes of length $p$, so that we need 
$	s=\left\lceil\frac{ N+2-a}{p}\right\rceil$
 descending jumps and corresponding factors. Finally, the power $n$ will be the sum of both the number of ascending jumps and the descending jumps, i.e. $n=r+s$. 
\end{proof}

\begin{teo}[Mixed multiple Gaussian quadrature formulas]
	The following Gauss quadrature formulas hold
	\begin{align}\label{eq:multiple_Gauss_quadrature}
		\int_\Delta x^n\d\psi_{b,a}(x)&=\sum_{k=1}^{N+1}\rho_{k,b}^{[N]}\mu_{k,a}^{[N]}\big( \lambda_k^{[N]}\big)^n, & 0&\leqslant n\leqslant d_{b,a}(N),
		& a&\in\{1,\dots,p\},& b&\in\{1,\dots,q\}.
	\end{align}
	Here the degrees of precision $d_{b,a}$ are optimal (for any power largest than $n$ a positive remainder appears, an exactness is lost). 
\end{teo}
\begin{proof}
	On the one hand, from Proposition \ref{pro:spectral_moments_Weyl} we have that $ 	\big(e_b^\xi\big)^\top T^n e^\nu_a=\int_\Delta x^n\d\psi_{b,a}(x)$. On the other hand, from Proposition \ref{pro:eq:moments_discrete}, we know that
	$	\big(e_b^\xi\big)^\top \big(T^{[N]}\big)^ne^\nu_a=\sum_{k=1}^{N+1}\rho_{k,b}^{[N]}\mu_{k,a}^{[N]}\big( \lambda_k^{[N]}\big)^n$. Hence, as we have
	\begin{align*}
		\big(e_b^\xi\big)^\top T^n e^\nu_a &=	\big(e_b^\xi\big)^\top \big(T^{[N]}\big)^n e^\nu_a, & 0&\leqslant n\leqslant d_{b,a}(N), & a&\in\{1,\dots,p\}, & b&\in\{1,\dots,q\}, 
	\end{align*}
	we get \eqref{eq:multiple_Gauss_quadrature}. Notice that for $n>d_{b,a}(N)$ a positive remainder will appear and exactness will be lost. 
	Indeed, observe that $T^n$ is oscillatory and that $e_a^\nu=\Lambda\mathcal A e_a$ and $e_b^\xi=\Upsilon\mathcal B e_a$ are positive vectors, so all the objects involved imply positive contributions.
\end{proof}
\begin{rem}
	In terms of the number of nodes or interpolation points $ \mathcal N=N+1$, the zeros of the characteristic polynomial $P_{N+1}$ or, equivalently the eigenvalues of $T^{[N]}$, we have that in the non multiple case for which $a=b=p=q=1$ the degree of precision is $2\mathcal N-1$ and we recover the well known Gauss quadrature formula, see for example \cite{Chihara, Ismail}. For the non mixed multiple situation we recover the result we got in discussed in \cite[Theorem 7]{PBF_1}, that is, degrees of precision $d_a=\mathcal N-1+\deg A^{(a)}_{\mathcal N-1}$.
\end{rem}

\begin{rem}
		Notice that for the standard orthogonality, i.e. $p=q=1$, the nodes are the zeros of an orthogonal polynomial of certain degree. This also happens for the non mixed multiple situation as the characteristic polynomials and one of the families of recursion polynomials, say $B_n$, coincide. However, for mixed multiple orthogonality the nodes are the zeros of the characteristic polynomial of the corresponding truncation, which is not an orthogonal polynomial. Consequently, the nodes are not, in general, the zeros of the left or right recursion polynomials, that are the ones satisfying the mixed multiple orthogonal relations.
\end{rem}

\begin{rem}
	A quadrature is said to be interpolatory if there is a polynomial that interpolates the function for which a weighted integral is supposed to be approximated by a quadrature. In the non mixing multiple orthogonal quadrature the interpolation polynomial is $P_N=B_N$ for all the measures $\d\psi_a$, $a\in\{1,\dots,p\}$. Now, for the mixed multiple orthogonality for each $b\in\{1,\dots,q\}$, we use the interpolation polynomials $B^{(b)}_N$ for the measures $\d\psi_{b,a}$, $a\in\{1,\dots,p\}$, so that in order to have an interpolatory quadrature we need the degrees of precision to be at least $\deg B^{(b)}_N-1$, which in fact is the case.
\end{rem}

\begin{rem}
	For the case $p=q$, i.e. we are dealing with the usual matrix orthogonality, as we are working with $p\times p$ blocks we take $\mathcal N=Mp$, with $M\in\N$, the degree of precision given in Proposition \ref{pro:degrees_precision} must be the smaller degree of precision in $p\times p$ block, i.e. 
		\begin{align*}
		d(N)&= 2 \deg B^{(p)}_N+1=2\left\lceil\frac{ \mathcal N+1-p}{p}\right\rceil-1=2\left\lceil\frac{ (M-1)p+1}{p}\right\rceil-1=2(M-1)+2-1=2M-1.
	\end{align*}
This is the optimal degree of precision according to Durán and Polo \cite{Duran_Polo}.
	
\end{rem}

\section*{Acknowledgments}
AB acknowledges Centro de Matemática da Universidade de Coimbra UID/MAT/00324/2020, funded by the Portuguese Government through FCT/MEC and co-funded by the European Regional Development Fund through the Partnership Agreement PT2020.

AF acknowledges CIDMA Center for Research and Development in Mathematics and Applications (University of Aveiro) and the Portuguese Foundation for Science and Technology (FCT) within project UIDB/MAT/UID/04106/2020 and UIDP/MAT/04106/2020.

MM acknowledges Spanish ``Agencia Estatal de Investigación'' research projects [PGC2018-096504-B-C33], \emph{Ortogonalidad y Aproximación: Teoría y Aplicaciones en Física Matemática} and [PID2021- 122154NB-I00], \emph{Ortogonalidad y Aproximación con Aplicaciones en Machine Learning y Teoría de la Probabilidad}.

	\end{document}